\documentclass[11pt]{amsart}

\usepackage[T1]{fontenc}
\usepackage[utf8]{inputenc}
\usepackage{textcomp}
\usepackage{lmodern}
\usepackage{microtype}
\usepackage{graphicx}

\usepackage{amsmath}
\usepackage{amssymb}
\usepackage{amsthm}
\usepackage{amscd}
\usepackage{mathabx}

\usepackage{hyperref}

\usepackage{mathscinet}
\usepackage[giveninits=true,doi=false,url=false,sortcites,backend=biber,backref=true]{biblatex}
\usepackage{bussproofs}

\newtheorem{theorem}{Theorem}[section]

\newtheorem{lemma}[theorem]{Lemma}

\theoremstyle{definition}
\newtheorem{definition}[theorem]{Definition}

 \def \dom{\operatorname{dom}}
 \def \rng{\operatorname{rng}}

\mathchardef\mhyphen="2D

\allowdisplaybreaks[2]

\DeclareMathOperator{\lleq}{\underline{\ll}}

\newcommand{\bOmega}{\underline{\mathtt{\Omega}}}

\begin{document}

\title{Proofs that Modify Proofs, 1/2\\{{Work in Progress}}}
\author{Henry Towsner}
\date{\today}
\thanks{Partially supported by NSF grant DMS-2054379\\
I'm grateful to Isaiah Hilsenrath, Carmine Ingram, Hanul Jeon, Oualid Merzouga, Alvaro Pintado, and Jin Wei for many helpful suggestions.}
\address {Department of Mathematics, University of Pennsylvania, 209 South 33rd Street, Philadelphia, PA 19104-6395, USA}
\email{htowsner@math.upenn.edu}
\urladdr{\url{http://www.math.upenn.edu/~htowsner}}

\begin{abstract}
  This paper is a prelude and elaboration on Proofs that  Modify Proofs. Here we present an ordinal analysis of a fragment of the $\mu$-calculus around the strength of parameter-free $\Pi^1_2$-comprehension using the same approach as that paper, interpreting functions on proofs as proofs in an expanded system.

  We build up the ordinal analysis in several stages, beginning by illustrating the method systems at the strength of  paremeter-free $\Pi^1_1$-comprehension and full $\Pi^1_1$-comprehension.
\end{abstract}

\maketitle

\section{Introduction}

Our goal in this paper is to give an ordinal analysis of the arithmetic $\mu$-calculus---the extension of the theory of first-order arithmetic by operators allowing the construction of fixed points. This theory is known \cite{MR1625856,MR1777769,Mollerfeld} to be proof-theoretically equivalent to the theory of $\Pi^1_2\mhyphen\mathrm{CA}_0$. Existing analyses of systems around this strength are notoriously difficult \cite{RathjenClaim,MR3443524,RathjenParameterFree}, and the extensions to the full theory remain unpublished.

The methods here are somewhat different, building on the approach in \cite{towsner:MR2499713}, and this paper is intended to serve as a preview of the techniques in \cite{2403.17922}, where the methods here are extended to full second order arithmetic.

To introduce the method, we build it up in four steps. First, we describe our approach for the theory $\mathsf{ID}_1$ which allows a single, unnested inductive definition. This approach is morally the same as, and inspired by, the approach in \cite{MR1943302}. However we introduce one modification, which at this stage is a minor technical change---instead of using an $\Omega$ rule branching over proofs, we use our new $\Omega^\flat$ rule and its companion $\mathrm{Read}$ rule, which allow us to represent functions on proofs as ill-founded ``proofs'' (which we call ``proof trees'').

This gives us the chance to illustrate ordinal bounds for these ill-founded proofs---we are able to assign ordinal \emph{terms} involving free variables as ``bounds'' for these ill-founded proofs, corresponding to the fact that our functions should map well-founded proofs to well-founded proofs. The ``uncountable'' ordinal notation $\Omega$ which typically appears in notations for the Howard--Bachmann ordinal \cite{MR3525540,MR0329869}, the proof-theoretic ordinal of this theory, shows up as a sort of ``meta-variable''.

We then extend this to the theory $\mathsf{ID}_{<\omega}$ of iterated inductive definitions (\cite{MR655036} is the comprehensive treatment of this and similar theories). Our main point here is to illustrate that our functions on proofs are able to ``lift'' to larger domains, allowing us to almost immediately extend the analysis of $\mathsf{ID}_1$.

Our third step, perhaps the most important, extends to a theory of nested inductive definitions $\mu xX.\phi(x,X,\mu yY.\psi(x,y,X,Y))$ where $X$ may appear \emph{negatively} in $\psi$ and $\mu yY.\psi(x,y,X,Y)$ appears negatively in $\phi$. This extends out of the realm of iterated inductive definitions towards the full $\mu$-calculus. Here we illustrate what the complications of the previous sections really get us: since we are representing functions on proofs as proof trees, we are able to write down functions on functions on proofs as yet more proof trees. For ordinal bounds, we need to use variables of function sort, representing the fact that our functions need to map functions which preserve well-foundedness to functions which preserve well-foundedness. Our ordinal notations are based on the one introduced in \cite{towsner2025polymorphicordinalnotations}.

\section{The $\mu$-Calculus}

\subsection{Language}

All theories in this paper will be in the language $\mathcal{L}_{\mu}$, the arithmetic $\mu$-calculus, which extends the language of first-order arithmetic by additional predicates for fixed points of positive formulas.

\begin{definition}
  The symbols of $\mathcal{L}_{\mu}$ consist of:
  \begin{itemize}
  \item an infinite set $Var=\{x,y,z,\ldots\}$ for first-order variables,
  \item a constant symbol $0$,
  \item a unary function symbol $\mathrm{S}$,
  \item for every $n$-ary primitive recursive relation $R$, an $n$-ary predicate symbol $\dot{R}$,
  \item an infinite set $Var_2=\{X,Y,Z,\ldots\}$ for second-order variables,
  \item the unary connective $\neg$,
  \item two binary connectives, $\wedge$ and $\vee$,
  \item two quantifiers, $\forall$ and $\exists$,
  \item a fixed-point quantifier $\mu$
  \end{itemize}
  where all symbols are distinct.

  The terms are given by:
  \begin{itemize}
  \item Each first-order variable is a term,
  \item $0$ is a term,
  \item if $t$ is a term, $St$ is a term.
  \end{itemize}

  In order to define formulas, we need to define atomic formulas as well as, simultaneously, the notion of a second-order variable appearing positively or negatively. We define the atomic formulas as follows:
  \begin{itemize}
  \item whenever $\dot{R}$ is a symbol for an $n$-ary primitive recursive relation and $t_1,\ldots,t_n$ are terms, there is an atomic formula $\dot{R}t_1\cdots t_n$,
  \item whenever $X$ is a second order variable and $t$ is a term, $Xt$ is an atomic formula,
  \item whenever $\phi$ is a formula, $x$ is a first-order variable, $X$ is a second-order variable, $t$ is a term, and $X$ does not appear negatively in $\phi$ then $\mu xX\phi t$ is an atomic formula which we usually write $t\in\mu xX.\phi$.
  \end{itemize}

  We define formulas as follows:
  \begin{itemize}
  \item whenever $\sigma$ is an atomic formula, $\sigma$ and $\neg\sigma$ are formulas, which we call \emph{literals}, and we write $t\not\in \mu xX.\phi$ in place of $\neg t\in \mu xX. \phi$,
  \item whenever $\phi$ and $\psi$ are formulas, $\phi\wedge\psi$ and $\phi\vee\psi$ are formulas,
  \item whenever $\phi$ is a formula and $x$ is a first-order variable, $\forall x\,\phi$ and $\exists x\,\phi$ are formulas.
  \end{itemize}

  We define the second-order variables appearing positively (resp. negatively) in a formula, $pos(\phi)$ (resp. $neg(\phi)$) as follows:
  \begin{itemize}
  \item $pos(\dot{R}t_1\cdots t_n)=pos(\neg \dot{R}t_1\cdots t_n)=neg(\dot{R}t_1\cdots t_n)=neg(\neg \dot{R}t_1\cdots t_n)=\emptyset$,
  \item $pos(Xt)=\{X\}$, $neg(Xt)=\emptyset$,
  \item $pos(\neg Xt)=\emptyset$, $neg(\neg Xt)=\{X\}$,
  \item $pos(t\in\mu yY. \phi)=pos(\phi)\setminus\{Y\}$,
  \item $neg(t\in\mu yY.\phi)=neg(\phi)$,
  \item $pos(t\not\in\mu yY.\phi)=neg(\phi)$,
  \item $neg(t\not\in\mu yY.\phi)=pos(\phi)\setminus\{Y\}$,
  \item $pos(\phi\wedge\psi)=pos(\phi\vee\psi)=pos(\phi)\cup pos(\psi)$,
  \item $neg(\phi\wedge\psi)=neg(\phi\vee\psi)=neg(\phi)\cup neg(\psi)$,
  \item $pos(\forall x\,\phi)=pos(\exists x\,\phi)=pos(\phi)$,
  \item $neg(\forall x\,\phi)=neg(\exists x\,\phi)=neg(\phi)$.
  \end{itemize}
\end{definition}

\begin{definition}
  We define the negation of a formula, ${\sim}\phi$, inductively by:
  \begin{itemize}
  \item if $\sigma$ is an atomic formula, ${\sim}\sigma$ is $\neg\sigma$ and ${\sim}\neg\sigma$ is $\sigma$,
  \item ${\sim}(\phi\wedge\psi)$ is $({\sim}\phi)\vee({\sim}\psi)$,
  \item ${\sim}(\phi\vee\psi)$ is $({\sim}\phi)\wedge({\sim}\psi)$,
  \item ${\sim}\forall x\,\phi$ is $\exists x\,{\sim}\phi$,
  \item ${\sim}\exists x\,\phi$ is $\forall x\,{\sim}\phi$.
  \end{itemize}
\end{definition}

We define the free variables in a formula, as well as the notion of substituting for a free variable, in the usual way.
\begin{definition}
  We define the free variables, $\mathrm{FV}$ inductively by:
  \begin{itemize}
  \item $\mathrm{FV}(0)=\emptyset$,
  \item $\mathrm{FV}(St)=\mathrm{FV}(t)$,
  \item $\mathrm{FV}(\dot{R}t_1\cdots t_n)=\bigcup_{i\leq n}\mathrm{FV}(t_i)$,
  \item $\mathrm{FV}(Xt)=\mathrm{FV}(t)$,
  \item $\mathrm{FV}(t\in\mu xX. \phi)=\mathrm{FV}(t)\cup(\mathrm{FV}(\phi)\setminus\{x\})$,
  \item $\mathrm{FV}(\neg\sigma)=\mathrm{FV}(\sigma)$,
  \item $\mathrm{FV}(\phi\wedge\psi)=\mathrm{FV}(\phi\vee\psi)=\mathrm{FV}(\phi)\cup\mathrm{FV}(\psi)$,
  \item $\mathrm{FV}(\forall x\,\phi)=\mathrm{FV}(\exists x\,\phi)=\mathrm{FV}(\phi)\setminus\{x\}$.
  \end{itemize}

  Similarly, we define the free second order variables, $\mathrm{FV}_2$ inductively by:
  \begin{itemize}
  \item $\mathrm{FV}_2(\dot{R}t_1\cdots t_n)=\emptyset$,
  \item $\mathrm{FV}_2(Xt)=\{X\}$,
  \item $\mathrm{FV}_2(t\in\mu xX.\phi)=\mathrm{FV}_2(\phi)\setminus\{X\}$,
  \item $\mathrm{FV}_2(\neg\sigma)=\mathrm{FV}_2(\sigma)$,
  \item $\mathrm{FV}_2(\phi\wedge\psi)=\mathrm{FV}_2(\phi\vee\psi)=\mathrm{FV}_2(\phi)\cup\mathrm{FV}_2(\psi)$,
  \item $\mathrm{FV}_2(\forall x\,\phi)=\mathrm{FV}_2(\exists x\,\phi)=\mathrm{FV}_2(\phi)\setminus\{x\}$.
  \end{itemize}

  We say $\phi$ is \emph{closed} if $\mathrm{FV}_1(\phi)=\mathrm{FV}_2(\phi)=\emptyset$.
  
  Let $S\subseteq Var$ and let $F$ map $S$ to terms. For a term $t$, we define $t[S\mapsto F]$ inductively by:
  \begin{itemize}
  \item if $x\in S$ then $x[S\mapsto F]$ is $F(x)$,
  \item if $x\not\in S$ then $x[S\mapsto F]$ is $x$,
  \item $0[S\mapsto F]$ is $0$,
  \item $(\mathrm{S}t)[S\mapsto F]$ is $\mathrm{S}(t[S\mapsto F])$,
  \end{itemize}

  Let $S\subseteq Var\cup Var_2$ and let $F$ be a function mapping $S\cap Var$ to terms and $S\cap Var_2$ to pairs $(\psi,y)$ where $\psi$ is a formula and $y\in Var$. For a formula $\phi$, we define $\phi[S\mapsto F]$ inductively by:
  \begin{itemize}
  \item $(\dot{R}t_1\cdots t_n)[S\mapsto F]$ is $\dot{R}(t_1[S\cap Var\mapsto F]\cdots t_n[S\cap Var\mapsto F])$,
  \item if $X\in S$ with $F(X)=(\psi,y)$ then $(Xt)[S\mapsto F]$ is $\psi[y\mapsto t[S\cap Var\mapsto F]]$,
  \item if $X\not\in S$ then $(Xt)[S\mapsto F]$ is $X(t[S\mapsto F])$,
  \item $(\neg\sigma)[S\mapsto F]$ is $\neg(\sigma[S\mapsto F])$,
  \item $(\phi\wedge\psi) [S\mapsto F]$ is $(\phi [S\mapsto F])\wedge(\psi [S\mapsto F])$,
  \item $(\phi\vee\psi) [S\mapsto F]$ is $(\phi [S\mapsto F])\vee(\psi [S\mapsto F])$,
  \item $(\forall x\,\phi)[S\mapsto F]$ is $\forall x\,(\phi[S\setminus\{x\}\mapsto F])$,
  \item $(\exists x\,\phi)[S\mapsto F]$ is $\exists x\,(\phi[S\setminus\{x\}\mapsto F])$,
  \item $(t\in \mu xX.\phi)[S\mapsto F]$ is $t[S\mapsto F]\in \mu xX.\phi[S\setminus\{x,X\}\mapsto F]$.
  \end{itemize}

  As usual, we only wish to allow substitutions when no variables are captured by a quantifier (continuing to understand $\mu xX.\phi$ as binding both $x$ and $X$). We define when \textit{$F$ is substitutable for $S$ in $\phi$} by:
  \begin{itemize}
  \item $F$ is always substitutable for $S$ in $\dot{R}t_1\cdots t_n$,
  \item $F$ is substitutable for $S$ in $Xt$ if $X\not\in S$ or $X\in S$, $F(X)=(\psi,y)$, and $t[S\mapsto F]$ is substitutable for $y$ in $\psi$,
  \item $F$ is substitutable for $S$ in $t\in \mu xX.\phi$ if
    \begin{itemize}
    \item $F$ is substitutable for $S\setminus\{x,X\}$ in $\phi$,
    \item for all $y\in S\cap \mathrm{Free}_1(\phi)\setminus\{x\}$, $x\not\in\mathrm{Free}_1(F(y))$, and
    \item for all $Y\in S\cap \mathrm{Free}_2(\phi)\setminus\{X\}$ with $F(Y)=(\psi,y)$, we have $x\not\in \mathrm{Free}_1(\psi)$ and $X\not\in \mathrm{Free}_2(\psi)$,
    \end{itemize}
  \item $F$ is substitutable for $S$ in $\neg\sigma$ if and only if $F$ is substitutable for $S$ in $\sigma$,
  \item $F$ is substitutable for $S$ in $\phi\wedge\psi$ and $\phi\vee\psi$ if and only if $F$ is substitutable for $S$ in both $\phi$ and $\psi$,
  \item $F$ is substitutable for $S$ in $\forall x,\phi$ and $\exists x\,\phi$ if and only if $F$ is substitutable for $S\setminus\{x\}$ in $\phi$ and, for each $y\in S\cap \mathrm{Free}_1(\phi)\setminus\{x\}$, $x\not\in\mathrm{Free}_1(F(y))$.
  \end{itemize}

  We often say ``$\phi[S\mapsto F]$ is a permitted substitution'' when $F$ is substitutable for $S$ in $\phi$.
  
\end{definition}
We refer to the bare form $\mu xX.\phi$, with no $n\in$ or $n\not\in$, as a \emph{$\mu$-expression}.

\begin{definition}
  A formula $\phi$ is \emph{proper} if $\mathrm{FV}_2(\phi)=\emptyset$.
\end{definition}

We need several ways to measure the complexity of a formula. We define the rank of a formula in the usual way, as the height of the construction tree corresponding to the formula, where we understand literals (even of the form $t\in\mu xX.\phi$ or $t\not\in\mu xX.\phi$) to have rank $0$:
\begin{definition}
  We define the rank of a formula, $rk(\phi)$, inductively:
  \begin{itemize}
  \item when $\phi$ is a literal, $rk(\phi)=0$,
  \item $rk(\phi\wedge\psi)=rk(\phi\vee\psi)=\max\{rk(\phi),rk(\psi)\}+1$,
  \item $rk(\forall x\,\phi)=rk(\exists x\,\phi)=rk(\phi)+1$.
  \end{itemize}
\end{definition}

Next, we want to measure how complicated $\mu$-expressions are. This is a bit complicated---the motivating example is the following related pair of expressions:
\[\mu xX. \phi(x,X,\mu yY. \psi(x,X,y,Y))\]
and
\[\mu yY. \psi(n, \mu xX. \phi(x,X,\mu yY. \psi(x,X,y,Y)), y, Y).\]
If we write $A(x,X)$ for $\mu yY. \psi(x,X,y,Y)$ and $B$ for $\mu xX.\phi(x,X,A(x,X))$, then the second expression is precisely $A(n,B)$.

In one sense, $A(n,B)$ is clearly the more complicated expression, since it literally contains $B$. On the other hand, the actual construction of $A(n,B)$, from previously existing things, is simpler---$A(n,B)$ is simply a fixed point of a formula which happens to have $B$ as a constant, while constructing $B$ requires a more complicated iterated fixed point construction.

To make sense of this, it will help to remove the set parameters from formulas---that is, given a formula $\phi(X)$ with a distinguished free variable $X$, we wish to write it in the form
\[\phi'(X,Z_1,\ldots,Z_n)[Z_i\mapsto\psi_i]\]
where none of the $\psi_i$ contain $X$ free.

\begin{definition}

  Let $\phi$ be a formula. The \emph{parameterized form} of $\phi(X)$ is a formula $\phi'$ together with fresh free variables $Z_1,\ldots,Z_n$ and formulas $\psi_1,\ldots,\psi_n$ such that:
  \begin{itemize}
  \item each $\psi_i(y)$ is a formula of the form $y\in\mu zZ. \psi'_i$ with $\mu zZ.\psi'_i$ closed\footnote{The essential point is that $\psi'_i$ has no second-order free variables. It is easier, and costs us no generality, to assume it has no first-order free variables as well.}
  \item $\phi$ is $\phi'[Z_i\mapsto\psi_i]$,
  \item the substitution $\phi'[Z_i\mapsto \psi_i]$ is permissible,
  \item $\phi'$ does not contain any sub-formula of the form $y\in\mu zZ. \psi'$ where $\mu zZ. \psi'$ is closed.
  \end{itemize}

\end{definition}

We now define the depth of a formula. Our notation is new and potentially confusing, so we briefly explain it. What we really want to do is assign a measure of complexity, which we call \emph{depth}, to $\mu$-expressions; this naturally leads to defining the depth of a formula as the largest depth of a $\mu$-expression appearing in it.

For us, the depth of a formula will be an expression $\bOmega^k\cdot n_k+\cdots+\bOmega^1\cdot n_1+n_0$, where $k,n_k,\ldots,n_0$ are finite numbers. These are added like ordinal numbers, where $\bOmega$ is taken to be a limit ordinal.

Finite depth formulas describe unnested $\mu$-expressions. A formula has depth $0$ when it contains no $\mu$-expressions; the simplest $\mu$-expressions, those of the form $\mu xX.\phi$ where $\phi$ has depth $0$ then have depth $1$. An expression $\mu xX.\phi(\mu yY.\psi)$ where $\psi$ does not contain $X$ free, in turn, has depth $2$.

As soon as we encounter nesting, we jump to ``uncountable'' depths: an expression $\mu xX.\phi (\mu yY.\psi(X))$ has depth $\bOmega$ (as does an expression $\mu xX.\phi(\mu yY.\psi(\mu zZ.\rho(X)))$). The reason for the gap is that formulas of depth $\alpha$, for $\alpha$ an infinite countable ordinal, do exist in other languages (see \cite{MR655036}), and indeed sit between finite depth and our depth $\bOmega$ formulas in complexity. This will not be explored in this paper, however---for our purposes, $\bOmega$ and $\omega$ are interchangeable.

\begin{definition}\
When $\phi$ is a formula and $\mathcal{S}$ is a set of variables, $\mathrm{nest}(\phi,\mathcal{S})$ is defined to be the maximum of $0$ and $\mathrm{nest}(\psi,\mathcal{S}\cup\{Y\})+1$ where $\psi$ ranges over subformulas of $\phi$ of the form $\mu yY. \psi'$ such that $\mathrm{Free}_2(\mu yY.\psi')\cap\mathcal{S}\neq\emptyset$.

Then we set
\begin{itemize}
\item $dp(\dot{R}t_1\cdots t_n)=dp(\neg \dot{R}t_1\cdots t_n)=dp(Xt)=dp(\neg Xt)=0$,
\item if the parameterized form of $\phi(X)$ is $\phi'[Z_i\mapsto\psi_i]$ then $dp(t\in\mu xX.\phi)=dp(t\not\in\mu xX.\phi)=\max_idp(\psi_i)+\bOmega^{\mathrm{nest}(\phi,\{X\})}$,
\item $dp(\phi\wedge\psi)=dp(\phi\vee\psi)=\max\{dp(\phi),dp(\psi)\}$,
\item $dp(\forall x\,\phi)=dp(\exists x\,\phi)=dp(\phi)$.
\end{itemize}  
\end{definition}

In our example above, 
\[\mu xX. \phi(x,X,\mu yY. \psi(x,X,y,Y))\]
has nesting $2$ and depth $\bOmega$ while
\[\mu yY. \psi(n, \mu xX. \phi(x,X,\mu yY. \psi(x,X,y,Y)), y, Y)\]
has nesting $1$ and depth $\bOmega+1$.

As a final note, we mention that the ``ordinals'' appearing as the depth of formula never gets conflated with an ordinal from our ordinal notations systems. (For instance, we will never take some complicated ordinal notation $\alpha$ and then consider a formula of depth $\alpha$.) For this reason, we write the symbol appearing in depths as $\bOmega$, while later we will use $\Omega$ in our ordinal notations with a similar meaning but in a different context. (Our use of the notation $\bOmega$ \emph{is} suggestive, though it's not entirely clear of what---it is not clear whether a sensible definition can be found for formulas of depth $\bOmega^\omega$ or $\bOmega^{\bOmega}$ and so on.)

\subsection{Inference Rules}

It will be helpful to recall some standard definitions for working with proof systems, especially since we will later have to complicate the familiar versions of these.

\begin{definition}
  A \emph{sequent} is a finite set of formulas.
\end{definition}

\begin{definition}
  An \emph{inference rule} is a tuple $\mathcal{R}=(I,\Delta,\{\Delta_\iota\}_{\iota\in I})$ where $I$ is a set, $\Delta$ is a sequent, and $\{\Delta_\iota\}_{\iota\in I}$ is a set of sequents indexed by $I$.
\end{definition}
We usually write inference rules using the notation

\AxiomC{$\cdots$}
\AxiomC{$\Delta_\iota$}
\AxiomC{$\cdots\ (\iota\in I)$}
\LeftLabel{$\mathcal{R}$}
\TrinaryInfC{$\Delta$}
\DisplayProof

The label $\mathcal{R}$ is the name of the inference rule, the set $I$ is the \emph{premises of $\mathcal{R}$}, the sequent $\Delta$ is the \emph{conclusion sequent}, and the sequent $\Delta_\iota$ is the \emph{premise sequent at $\iota$}. When $\mathcal{R}$ is an inference rule, we write $|\mathcal{R}|$ for its premises and $\Delta(\mathcal{R})$, $\Delta_\iota(\mathcal{R})$ for the corresponding sequents.

Unless explicitly noted, when $|I|=1$ we assume $I=\{\top\}$, and when $|I|=2$ we assume $I=\{\mathrm{L},\mathrm{R}\}$.

It is our assumption that the sets $I$, across all rules ever considered, are disjoint---that is, given $\iota\in |\mathcal{R}|$, we can recover the rule $\mathcal{R}$ from $\iota$, and we denote it $\mathcal{R}(\iota)$.

We will often equivocate notationally, writing rules whose sets of premises are $\mathbb{N}$, $\{\mathrm{L},\mathrm{R}\}$, or $\{\top\}$. We always take it as given that the set of premises is \emph{really} some set of pairs from $\{t\}\times\mathbb{N}$ or $\{t\}\times\{0,1\}$ or the like, where $t$ is an encoding of the rule.

For reasons that will become clear, our perspective on deductions emphasizes, not that they are built from axioms in a well-founded way, but that they are ``read'' from the root up. In particular, our basic object is a \emph{proof tree}, the co-well-founded analog of a proof.

There is one rule which, for technical reasons, we wish to include in all theories:

\AxiomC{$\emptyset$}
\LeftLabel{Rep}
\UnaryInfC{$\emptyset$}
\DisplayProof

\begin{definition}
  A \emph{proof system} is a set of inference rules which includes the Rep rule.

  When $\mathfrak{T}$ is a proof system, a \emph{proof tree in $\mathfrak{T}$} is a function $d$ where:
  \begin{itemize}
  \item the range of $d$ is $\mathfrak{T}$,
  \item $\langle\rangle\in\dom(d)$,
  \item if $\sigma\in\dom(d)$ and $\iota\in |d(\sigma)|$ then $\sigma\iota\in\dom(d)$.
  \end{itemize}

  For any $\sigma\in\dom(d)$, we define $\Gamma(d,\sigma)$ to be the set of $\phi$ such that there is some $\tau$ with $\sigma\tau\in\dom(d)$, $\phi\in\Delta(d(\sigma\tau))$, and for all $\tau'\iota\sqsubseteq\tau$, $\phi\not\in\Delta_\iota(d(\sigma\tau'))$. We write $\Gamma(d)$ for $\Gamma(d,\langle\rangle)$.

  We write $d\vdash\Sigma$ to mean $\Gamma(d)\subseteq\Sigma$.

  A \emph{$\mathfrak{T}$-position} is an element of $\dom(d)$ for some $d$. Equivalently, a $\mathfrak{T}$-position is a sequence from $\bigcup_{\mathcal{R}\in\mathfrak{T}}|\mathcal{R}|$.
\end{definition}
In general, proof trees might not be well-founded, and we will have to restrict to proof trees satisfying certain commitments about well-foundedness.

When a proof tree is well-founded, the object we usually call a proof might be more naturally identified with the image of the function $d$---the image of $d$ is a (in this case well-founded) tree of rules. It is not hard to check that our definition of $\Gamma$ satisfies the usual inductive property of the endsequent,
\[\Gamma(d,\sigma)=\Delta(d(\sigma))\cup\bigcup_{\iota\in|d(\sigma)|} \Gamma(d,\sigma\iota)\setminus\Delta_\iota(d(\sigma)).\]
It then follows by induction that when $d$ is well-founded, $\Gamma(d)$ is the conclusion of the proof corresponding to $d$ as usually defined.

\begin{definition}
  A proof tree is a \emph{deduction} if it is well-founded.
\end{definition}

A standard complication is that some rules have an \emph{eigenvariable} condition.

\begin{definition}
A \emph{proof system with eigenvariable restrictions} is a proof system $\mathfrak{T}$ together with a partial function $Eig:\mathfrak{T}\rightarrow Var$.

When $\mathfrak{T}$ is a theory with eigenvariable restrictions, a proof tree $d$ is valid if, for every $\sigma\tau\in\dom(d)$, either $Eig(d(\sigma))$ is undefined or $Eig(d(\sigma))$ does not appear free in any formula in $\Gamma(d(\sigma\iota))$ for any $\iota\in|d(\sigma)|$.
\end{definition}
When dealing with theories with eigenvalue restrictions, we always assume deductions are valid.

Following \cite{MR1943302}, we write $!y!$ besides a rule $\mathcal{R}$ to indicate that $Eig(\mathcal{R})=y$.

\subsection{The Theory $\mu$}

Our eventual concern is the finitary theory $\mu$-arithmetic with eigenvariable conditions, which extends first-order arithmetic by adding rules stating that each $\mu xX. \phi$ is the least fixed point of the monotone operator
\[S\mapsto \{n\mid \phi(n,S)\}.\]
This is represented by two rules, a closure rule, which says that if $\phi(t,\mu xX. \phi)$ holds then $t\in \mu xX.\phi$; and an induction rule, which says that $\mu xX. \phi$ is the smallest set with this property.

Formally, to produce a proof system for $\mu$-arithmetic, we need the following rules:

\bigskip

\AxiomC{}
\LeftLabel{$\mathrm{Def}_\phi$}
\UnaryInfC{$\phi$}
\DisplayProof
where $\phi$ is a quantifier-free (but possibly with free variables) defining axiom for some $n$-ary primitive recursive relation $R$
\bigskip

\AxiomC{}
\LeftLabel{Ax$_{\eta}$}
\UnaryInfC{$\eta,\neg \eta$}
\DisplayProof
where $\eta$ is an atomic formula

\bigskip

\AxiomC{$\phi$}
\AxiomC{$\psi$}
\LeftLabel{$\wedge\mathrm{I}_{\phi\wedge\psi}$}
\BinaryInfC{$\phi\wedge\psi$}
\DisplayProof

\bigskip

\AxiomC{$\phi$}
\LeftLabel{$\vee\mathrm{I}^L_{\phi\vee\psi}$}
\UnaryInfC{$\phi\vee\psi$}
\DisplayProof
\quad\quad\quad
\AxiomC{$\psi$}
\LeftLabel{$\vee\mathrm{I}^R_{\phi\vee\psi}$}
\UnaryInfC{$\phi\vee\psi$}
\DisplayProof

\bigskip

\AxiomC{$\phi(y)$}
\LeftLabel{$\forall\mathrm{I}^y_{\forall x\phi}$}
\RightLabel{$!y!$, where $\phi[x\mapsto y]$ is a permitted substitution}
\UnaryInfC{$\forall x\,\phi$}
\DisplayProof

\bigskip

\AxiomC{$\phi(t)$}
\LeftLabel{$\exists\mathrm{I}^t_{\exists x\,\phi}$}
\RightLabel{where $\phi[x\mapsto t]$ is a permitted substitution}
\UnaryInfC{$\exists x\,\phi$}
\DisplayProof

\bigskip

\AxiomC{}
\LeftLabel{Ind$^t_{\forall x\,\phi}$}
\UnaryInfC{${\sim}\phi(0), \exists x\,(\phi(x)\wedge {\sim}\phi(Sx)), \phi(t)$}
\DisplayProof

\bigskip

\AxiomC{$\phi(t,\mu xX.\, \phi)$}
\LeftLabel{Cl$_{t\in\mu xX.\,\phi}$}
\RightLabel{where $\phi[x\mapsto t][X\mapsto \mu xX. \phi]$ is a permitted substitution}
\UnaryInfC{$t\in \mu xX.\,\phi$}
\DisplayProof

\bigskip

\AxiomC{}
\LeftLabel{Ind$^t_{\mu xX\,\phi, \psi(y)}$}
\UnaryInfC{$\exists x\,(\phi(x,\psi(y))\wedge {\sim}\psi(x)), t\not\in \mu xX\,\phi, \psi(t)$}
\DisplayProof

where $\phi[X\mapsto \psi]$, $\psi[y\mapsto x]$, and $\psi[y\mapsto t]$ are permitted substitutions

\bigskip

\AxiomC{$\phi$}
\AxiomC{${\sim}\phi$}
\LeftLabel{Cut$_{\phi}$}
\BinaryInfC{$\emptyset$}
\DisplayProof

\smallskip

Then $\mu$-arithmetic is the theory containing the rules Def, Ax, $\wedge$I, $\vee$I, $\forall$I, $\exists$I, Cl, Ind, and Cut for all choices of formulas and terms such that all formulas in $\Delta$ or $\Delta_\iota$ are proper. 

Recall \cite{MR1625856,MR1777769,Mollerfeld} that $\mu$-arithmetic is proof-theoretically equivalent (that is, proves the same $\Pi^1_1$ formulas) to the theory $\Pi^1_2$-\textrm{CA}$_0$ in the language of second-order arithmetic.

We will want to work with various fragments of this theory.

\begin{definition}
  For any $\alpha$, an \emph{$\mathsf{ID}_\alpha$ formula} is a formula which is proper and has depth $\leq\alpha$.

  $\mathsf{ID}_\alpha$ is the restriction of $\mu$-arithmetic to $\mathsf{ID}_\alpha$ formulas and $\mathsf{ID}_{<\alpha}$ is the restriction of $\mu$-arithmetic to formulas which are $\mathsf{ID}_\beta$ for some $\beta<\alpha$.
\end{definition}

\section{Proof Trees as Functions on Proof Trees}

In this section, we describe certain kinds of functions from proof trees to proof trees, namely those which are defined ``one rule at a time'' starting from the root. Furthermore, we extend our proof system so that we are able to represent these functions as proof trees in the extended system.

\begin{definition}
  We extend the language $\mathcal{L}_\mu$ to languages $\mathcal{L}^1_\mu$ and $\mathcal{L}^2_\mu$ by adding new formulas as follows. 
  \begin{itemize}
  \item every $\mathcal{L}_\mu$ formula is a $\mathcal{L}^1_\mu$ formula,
  \item if $\Theta$ is a sequent of $\mathcal{L}_\mu$ formulas, $\epsilon$ is a $\mathfrak{T}$-position for some $\mathfrak{T}$, and $n\in\mu xX.\phi$ is a formula in $\Theta$ then $[\Theta]^{n\in\mu xX.\phi,\epsilon}$ is a $\mathcal{L}^{1}_\mu$ formula,
  \item if $\Theta$ is a sequent of $\mathcal{L}^1_\mu$ formulas, $\epsilon$ is a $\mathfrak{T}$-position for some $\mathfrak{T}$, and $n\in\mu xX.\phi$ is a formula in $\Theta$ then $[\Theta]^{n\in\mu xX.\phi,\epsilon}$ is a $\mathcal{L}^{2}_\mu$ formula.
  \end{itemize}

\end{definition}

We could, of course, iterate this process to obtain higher iterations of bracketed formulas, but one of our goals in this paper is to avoid this complication. (We could also, more generally, consider $[\Theta]^{\Theta_0,\epsilon}$ for any sequent $\Theta_0$, although, again, we do not need this generalization in this paper.)

We interpret the formula $[\Theta]^{n\in\mu xX. \phi,\epsilon}$ as indicating that we are ``awaiting the input'' of a proof of $\Theta$ which appears at position $\epsilon$ in some proof of $n\in\mu xX.\phi$. (That is, we are expecting that $\Gamma(d,\epsilon)\supseteq\Theta$.)

\begin{definition}
  When $\epsilon$ is a $\mathfrak{T}$-position, we define $\Delta(\epsilon)$ by:
  \begin{itemize}
  \item $\Delta(\langle\rangle)=\emptyset$,
  \item $\Delta(\epsilon\iota)=\Delta_\iota(\mathcal{R}(\iota))$.
  \end{itemize}
\end{definition}
That is, $\Delta(\epsilon)$ is the topmost premise sequent in the path $\epsilon$.

The crucial rule is the following

\AxiomC{$\Delta(\mathcal{R})\setminus\Theta, \{[\Theta']^{n\in\mu xX.\phi,\epsilon\iota}\mid \iota\in|\mathcal{R}|, \Theta\subseteq\Theta'\}$}
\AxiomC{$(\mathcal{R}\in\mathfrak{T})$}
\LeftLabel{$\mathrm{Read}^{\mathfrak{T}}_{[\Theta]^{n\in\mu xX.\phi,\epsilon}}$}
\BinaryInfC{$\Delta(\epsilon)\setminus\Theta,[\Theta]^{n\in\mu xX.\phi,\epsilon}$}
\DisplayProof

Note that the premises of this rule are indexed \emph{by rules}. At a bare minimum, proof systems must be non-circular---if $\mathrm{Read}^{\mathfrak{T}}_{[\Theta]^{n\in\mu xX.\phi,\epsilon}}\in\mathfrak{T}'$ then we should be able to entirely construct the theory $\mathfrak{T}$ prior to constructing the theory $\mathfrak{T}'$.

Since it can be confusing to have rules appearing as indices of other rules, we adopt the convention that $|\mathrm{Read}^{\mathfrak{T}}_{[\Theta]^{n\in\mu xX.\phi,\epsilon}}|=\{\ulcorner \mathcal{R}\urcorner\mid\mathcal{R}\in\mathfrak{T}\}$---that is, when a rule appears as an index, we will always place it in $\ulcorner\cdot\urcorner$ braces.

We make the following assumptions about all proof systems we consider:
\begin{itemize}
\item $\mathfrak{T}$ contains the $\mathrm{Rep}$ rule,
\item if $\mathcal{R}\in\mathfrak{T}$ and $[\Theta]^{n\in\mu xX.\phi,\epsilon}\in\Delta(\mathcal{R})$ then $\mathcal{R}$ is a Read rule, and
\item if $\mathcal{R}\in\mathfrak{T}$, $[\Theta]^{n\in\mu xX.\phi,\epsilon}\in\Delta_\iota(\mathcal{R})$, and $\epsilon\neq\langle\rangle$ then $\mathcal{R}$ is a Read rule.
\end{itemize}

\begin{definition}
Let $\mathfrak{R}(n\in\mu xX.\phi,\mathfrak{T})$ be the theory consisting of all rules $\mathrm{Read}^{\mathfrak{T}}_{[\Theta]^{n\in\mu xX.\phi,\epsilon}}$ for any $\Theta$ containing $n\in\mu xX.\phi$ and any $\mathfrak{T}$-position $\epsilon$.

A \emph{locally defined function from $\mathfrak{T}$ to $\mathfrak{T}'$} is a pair $(F,\phi(x,X))$ where $F$ is a proof tree in $\mathfrak{T}'+\mathfrak{R}(n\in\mu xX.\phi,\mathfrak{T})$ such that:
  \begin{itemize}
  \item for any $\Theta,\epsilon$ so that $[\Theta]^{n\in\mu xX.\phi,\epsilon}\in\Gamma(F)$, we have $\epsilon=\langle\rangle$, and
  \item for any $\Theta,\mathfrak{T}'',\epsilon$ so that a rule $\mathrm{Read}^{\mathfrak{T}''}_{[\Theta]^{n\in\mu xX.\phi,\epsilon}}$ is in $\mathfrak{T}'$, $\mathfrak{T}''=\mathfrak{T}$.
  \end{itemize}

We usually just write $F$ for the locally defined function and let the formula $[n\in\mu xX.\phi]^{n\in\mu xX.\phi,\langle\rangle}$ be implicit. In this case we write $\phi_F$ for the formula $n\in\mu xX.\phi$.
\end{definition}

\begin{definition}
  Let $F$ be locally defined function from $\mathfrak{T}$ to $\mathfrak{T}'$ and let $d$ be a proof tree in $\mathfrak{T}$. We define a proof tree $\bar F(d)$ in $\mathfrak{T}'$ together with an auxiliary function $h:\dom(\bar F(d))\rightarrow \dom(F)$ by:
  \begin{itemize}
  \item $h(\langle\rangle)=\langle\rangle$,
  \item if $h(\sigma)$ is defined and $F(h(\sigma))$ is $\mathrm{Read}^{\mathfrak{T}}_{[\Theta]^{n\in\mu xX.\phi,\epsilon}}$ then $\bar F(d)(\sigma)=\mathrm{Rep}$ and $h(\sigma\top)=h(\sigma)\ulcorner d(\epsilon)\urcorner$,
  \item if $h(\sigma)$ is defined and $F(h(\sigma))$ is any other rule then $\bar F(d)(\sigma)=F(h(\sigma))$ and, for each $\iota\in|\bar F(d)(\sigma)|$, $h(\sigma\iota)=h(\sigma)\iota$.
  \end{itemize}
\end{definition}

The following lemma says that $\bar F(d)$ has the endsequent we expect. We need some technical conditions which basically say that the Read rules are placed in $F$ in a way that makes sense.
\begin{theorem}\label{thm:locally_def_endsequent}
  Fix a formula $n\in\mu xX.\phi$. If $F$ is a locally defined function from $\mathfrak{T}$ to $\mathfrak{T}'$ and $d$ is a proof tree in $\mathfrak{T}$ then
  \[\Gamma(\bar F(d))\subseteq (\Gamma(F)\setminus\{[n\in\mu xX.\phi]^{n\in\mu xX.\phi,\langle\rangle}\})\cup (\Gamma(d)\setminus\{n\in\mu xX.\phi\}).\]
\end{theorem}
That is, if $d\vdash\Delta, n\in\mu xX.\phi$ and $F\vdash\Sigma,[n\in\mu xX.\phi]^{n\in\mu xX.\phi,\langle\rangle}$ then $\bar F(d)\vdash\Delta,\Sigma$.
\begin{proof}
  Suppose $\psi\in\Gamma(\bar F(d))$, so there is some $\sigma\in \dom(\bar F(d))$ so that $\psi\in\Delta(\bar F(d)(\sigma))$ and $\psi\not\in\Delta_\iota(\bar F(d)(\tau))$ for any $\tau\iota\sqsubseteq\sigma$.

First, suppose $\psi\in\Gamma(F)$. Then it suffices to show $\psi\neq[n\in\mu xX.\phi]^{n\in\mu xX.\phi,\langle\rangle}$. But $\psi\in\Delta(\bar F(d)(\sigma))$ and $\bar F(d)(\sigma)$ cannot be a $\mathrm{Read}^{\mathfrak{T}''}_{[\Theta]^{n\in\mu xX.\phi,\epsilon}}$ rule, so $\psi\in\Gamma(F)\setminus\{[n\in\mu xX.\phi]^{n\in\mu xX.\phi,\langle\rangle}\}$.

So suppose $\psi\not\in\Gamma(F)$.  We have $\psi\in \Delta(F(h(\sigma)))$. Choose $\tau\sqsubseteq h(\sigma)$ minimal so that $\psi\in\Delta(F(\tau))$. Since $\psi$ is not removed at any rule of $\bar F(d)$, $\psi$ must be removed at some $\mathrm{Read}^{\mathfrak{T}}_{[\Theta]^{n\in\mu xX.\phi,\epsilon}}$ inference below $\tau$, so $\psi\in\Delta(\mathcal{R})\setminus\Theta$ where $\mathcal{R}$ is $d(\epsilon)$.

Suppose there is some $\epsilon'\iota\sqsubseteq\epsilon$ with $\psi\in\Delta_\iota(d(\epsilon'))$. Then there is some corresponding $\tau'\sqsubseteq \tau$ so that $F(\tau')$ is $\mathrm{Read}^{\mathfrak{T}}_{[\Theta']^{n\in\mu xX.\phi,\epsilon'}}$ for some $\Theta'\subseteq\Theta$. Therefore $\psi\not\in\Theta'$, so $\psi\in\Delta(F(\tau'))$, contradicting the minimality of $\tau$.

  Therefore $\psi\in\Gamma(d)\setminus\{n\in\mu xX.\phi\}$.
\end{proof}

It is sometimes more convenient to emphasize the coinductive construction of $\bar F$.

\begin{definition}
  When $d$ is a proof and $\sigma\in\dom(d)$, we write $d_\sigma$ for the sub-proof tree rooted at $\sigma$, given by $\dom(d_\sigma)=\{\tau\mid \sigma\tau\in\dom(d)\}$ and $d_\sigma(\tau)=d(\sigma\tau)$.
\end{definition}

When $F$ is a locally defined function, its sub-proof trees are typically not locally defined functions, since they depend on multiple inputs.

\begin{definition}
  We write $\Upsilon_{n\in\mu xX.\phi}(F)=\{\epsilon\mid \exists\Theta\, [\Theta]^{n\in\mu xX.\phi,\epsilon}\in\Gamma(F)\}$.
\end{definition}
We should think of $F$ as a function on a collection of sub-proofs $\{d_\epsilon\}_{\epsilon\in\Upsilon_{n\in\mu xX.\phi}(F)}$.

Let us add one more bit of notation---if $\mathcal{R}$ is a rule and $\{d_\iota\}_{\iota\in|\mathcal{R}|}$ is a collection of proof trees, we write $\mathcal{R}(\{d_i\})$ for the proof tree $d$ with $d(\langle\rangle)=\mathcal{R}$ and $d(\iota\tau)=d_\iota(\tau)$ for all $\iota\in|\mathcal{R}|$.

We can then say that
\[\bar F(\{d_{\epsilon}\})=\left\{\begin{array}{ll}
                                         \mathrm{Rep}(\bar F_{d_{\epsilon_0}(\langle\rangle)}(\{d_\epsilon\}_{\epsilon\in\Upsilon_{n\in\mu xX.\phi}(F)}\cup\{d_{\epsilon_0\iota}\}_{\iota\in|d_{\epsilon_0}(\langle\rangle)|}))&\text{if }F(\langle\rangle)=\mathrm{Read}^{\mathfrak{T}}_{[\Theta]^{n\in\mu xX.\phi,\epsilon_0}}\\
                                         \mathcal{R}(\{\bar F_\iota(\{d_\epsilon\})\}_\iota)&\text{if }F(\langle\rangle)=\mathcal{R}\neq \mathrm{Read}^{\mathfrak{T}}_{[\Theta]^{n\in\mu xX.\phi,\epsilon_0}}\end{array}\right..\]

\subsection{The $\Omega^\flat$ Rule}

We now further blur the difference between functions and proof trees by introducing a new rule, the $\Omega^\flat$ rule, which lets us place a function within a proof tree of a conventional sequent.

The $\Omega^\flat$ rule is

\AxiomC{$[n\in\mu xX. \phi]^{n\in\mu xX.\phi,\langle\rangle}$}
\LeftLabel{$\Omega^\flat_{n\not\in\mu xX.\phi}$}
\UnaryInfC{$n\not\in\mu xX.\phi$}
\DisplayProof

We need a partner rule that combines this rule with a Cut inference:

\AxiomC{$n\in\mu xX.\phi$}
\AxiomC{$[n\in\mu xX. \phi]^{n\in\mu xX.\phi,\langle\rangle}$}
\LeftLabel{Cut$\Omega^\flat_{n\not\in\mu xX.\phi}$}
\BinaryInfC{$\emptyset$}
\DisplayProof

Note that proof trees involving Read rules, including those which involve $\Omega^\flat$ and Cut$\Omega^\flat$ rules, tend to be ill-founded. In general, though, ill-founded proof trees are unsound (consider a proof with infinitely many Rep rules); we will later need to make sense of which proof trees are ``well-founded enough'' to consider.

Contrary to our usual convention, we take $|\Omega^\flat|=\{\bot\}$ and $|\mathrm{Cut}\Omega^\flat|=\{\top,\bot\}$ (with $\Delta_\top(\mathrm{Cut}\Omega^\flat)=\{n\in\mu xX.\phi\}$ and $\Delta_\bot(\mathrm{Cut}\Omega^\flat)=\{[n\in\mu xX. \phi]^{n\in\mu xX.\phi,\langle\rangle}\}$).

\section{Cut-Elimination for $\mathsf{ID}_1$}

\subsection{Infinitary Theory}

In this section, we consider the theory $\mathsf{ID}_1$---that is, the restriction to formulas where, in any $\mu$-expression $\mu xX.\phi$, the formula $\phi$ is an arithmetic formula (i.e. does not contain further $\mu$-expressions).

To prove cut-elimination we will introduce an infinitary theory. We need two other standard rules for infinitary theories, the True axiom and the $\omega$ rule:

\AxiomC{}
\LeftLabel{True$_{\eta}$}
\UnaryInfC{$\eta$}
\DisplayProof

where $\eta$ is a true closed literal of the form $\dot{R}t_1\cdots t_n$ or $\neg\dot{R}t_1\cdots t_n$

\bigskip

\AxiomC{$\cdots\phi(n)\cdots$}
\AxiomC{$(n\in\mathbb{N})$}
\LeftLabel{$\omega_{\forall x\,\phi}$}
\BinaryInfC{$\forall x\,\phi$}
\DisplayProof

The theory $\mathsf{ID}^{\infty,+}_1$ consists of
\begin{itemize}
\item the rules True, $\wedge\mathrm{I}$, $\vee\mathrm{I}^L$, $\vee\mathrm{I}^R$, $\omega$, $\exists\mathrm{I}$, Cl, and Rep for closed $\mathsf{ID}_1$ formulas.
\end{itemize}

The full theory $\mathsf{ID}_1^\infty$ adds:
\begin{itemize}
\item the Cut, $\Omega^\flat$, and Cut$\Omega^\flat$ rules for closed $\mathsf{ID}_1$ formulas,
\item the rules $\mathrm{Read}^{\mathsf{ID}_1^{\infty,+}}_{[\Theta]^{n\in\mu xX.\phi,\epsilon}}$ for $\Theta$ is a set of $\mathcal{L}_\mu$ formulas and $n\in\mu xX.\phi$ is a closed $\mathsf{ID}_1$ formula.
\end{itemize}

\begin{definition}
  When $\mathfrak{T}$ is a theory, we write $\mathfrak{T}_{<n}$ for the restriction in which all Cut rules have rank $<n$.
\end{definition}

In particular, $\mathsf{ID}^\infty_{1,<0}$ is the cut-free part of $\mathsf{ID}^\infty_1$. (Note that $\mathsf{ID}^\infty_{1,<0}$ does \emph{not} have a subformula property for all formulas because it still contains the Cut$\Omega^\flat$ rule.)

\subsection{Embedding}

We must show that the finitary theory can be embedded in the infinitary one. This also gives us an opportunity to illustrate how the $\Omega^\flat$ rule works, since we use it to replace the induction rule for a $\mu$-formula.

\begin{lemma}\label{thm:identity}
  For any $\mu$-expression $\mu xX.\phi$ there is a proof tree $d_{n\in\mu xX.\phi}$ in $\mathsf{ID}^\infty_{1,<0}$ with $\Gamma(d_{n\in\mu xX.\phi})\subseteq\{n\in\mu xX.\phi,n\not\in\mu xX.\phi\}$.
\end{lemma}
\begin{proof}
  We define $d_{n\in\mu xX.\phi}(\langle\rangle)=\Omega^\flat_{n\not\in\mu xX.\phi}$ and $d_{n\in\mu xX.\phi}(\langle\top\rangle)=\mathrm{Read}^{\mathsf{ID}_1^{\infty,+}}_{[n\in\mu xX.\phi]^{n\in\mu xX.\phi,\langle\rangle}}$.

  Suppose that, for some $\sigma$, we have set $d_{n\in\mu xX.\phi}(\sigma)=\mathrm{Read}^{\mathsf{ID}_1^{\infty,+}}_{[\Gamma]^{n\in\mu xX.\phi,\epsilon}}$. Then, for each $\mathcal{R}$, we set $d_{n\in\mu xX.\phi}(\sigma\ulcorner\mathcal{R}\urcorner)=\mathcal{R}$ and, for each $\iota\in|\mathcal{R}|$, $d_{n\in\mu xX.\phi}(\sigma\ulcorner\mathcal{R}\urcorner\iota)=\mathrm{Read}^{\mathsf{ID}_1^{\infty,+}}_{[\Gamma_\iota]^{n\in\mu xX.\phi,\epsilon\iota}}$ where $\Gamma_\iota$ is $\Gamma$ if $\Delta(\mathcal{R})\cap\Gamma=\emptyset$, and $\Gamma_\iota=\Gamma\cup\Delta_\iota(\mathcal{R})$ if $\Delta(\mathcal{R})\cap\Gamma\neq\emptyset$.
\end{proof}

\begin{lemma}\label{thm:ax}
  For every proper closed formula $\phi$, there is a proof tree $d_\phi$ in $\mathsf{ID}_{1,<0}^{\infty}$ with $\Gamma(d_\phi)\subseteq \{\phi,{\sim}\phi\}$.
\end{lemma}
\begin{proof}
  By induction on $\phi$. The case where $\phi$ is a literal is either immediate using $\mathrm{True}$ or is covered by the previous lemma. The cases where $\phi$ is constructed using $\wedge,\vee,\forall$, or $\exists$ are standard.
\end{proof}

\begin{lemma}\label{thm:substitution}
  For any formula $\psi(y)$ with only the free variable $y$, there is proof tree $\mathrm{Subst}^{n \in \mu xX. \phi\mapsto \psi(n)}$ in $\mathsf{ID}^{\infty}_{1,<0}$ with conclusion
  \[[n\in \mu xX.\phi]^{n\in\mu xX.\phi,\langle\rangle},\psi(n),\exists y\,\phi(y,\psi)\wedge{\sim}\psi(y).\]
\end{lemma}

The substitution function is fairly straightforward (and unchanged from the version that appears in, for instance, \cite{MR655036}); we wish to replace the predicate $n\in\mu xX.\phi$ with the formula $\psi$, and we do this inductively by looking at the bottommost rule of $d$. In order to make the induction go through, we need a stronger inductive hypothesis: we have some sequent $\Theta$ where $\mu xX.\phi$ appears positively and replace $\mu xX.\phi$ with $\psi$ for all the formulas in $\Theta$.

The main case is when $d$ ends with a $\mathrm{Cl}$ rule where the conclusion is some $n\in\mu xX.\phi$ appearing in $\Theta$; in this case we need to carry out the replacement

\AxiomC{$d'$}
\noLine
\UnaryInfC{$\vdots$}
\noLine
\UnaryInfC{$\phi(n,\mu xX.\phi)$}
\LeftLabel{$\mathrm{Cl}$}
\UnaryInfC{$n\in\mu xX.\phi$}
\DisplayProof
$\mapsto$
\AxiomC{$\mathrm{IH}(d',\Theta\cup\{\phi(n,\psi)\})$}
\noLine
\UnaryInfC{$\vdots$}
\noLine
\UnaryInfC{$\phi(n,\psi) ,\exists y\,\phi(y,\psi)\wedge{\sim}\psi(y)$}
\AxiomC{$\vdots$}
\noLine
\UnaryInfC{$\psi(n),{\sim}\psi(n)$}
\LeftLabel{$\wedge\mathrm{I}$}
\BinaryInfC{$\psi(n),\phi(n,\psi)\wedge{\sim}\psi(n) ,\exists y\,\phi(y,\psi)\wedge{\sim}\psi(y)$}
\LeftLabel{$\exists\mathrm{I}$}
\UnaryInfC{$\psi(n),\exists y\,\phi(y,\psi)\wedge{\sim}\psi(y)$}
\DisplayProof

The case where the conclusion of the final inference of $d$ does not appear in $\Theta$ is easy---we keep the final inference the same and apply the inductive hypothesis. When the the final inference is not a $\mathrm{Cl}$ rule but its conclusion appears in $\Theta$, we have to replace the rule with a suitable modification. For instance:

\AxiomC{$d_L$}
\noLine
\UnaryInfC{$\vdots$}
\noLine
\UnaryInfC{$\rho_L(\mu xX.\phi)$}
\AxiomC{$d_R$}
\noLine
\UnaryInfC{$\vdots$}
\noLine
\UnaryInfC{$\rho_R(\mu xX.\phi)$}
\BinaryInfC{$(\rho_L\wedge\rho_R)(\mu xX.\phi)$}
\DisplayProof
$\mapsto$
\AxiomC{$\mathrm{IH}(d_L,\Theta\cup\{\rho_L\})$}
\noLine
\UnaryInfC{$\vdots$}
\noLine
\UnaryInfC{$\rho_L(\psi)$}
\AxiomC{$\mathrm{IH}(d_R,\Theta\cup\{\rho_R\})$}
\noLine
\UnaryInfC{$\vdots$}
\noLine
\UnaryInfC{$\rho_R(\psi)$}
\BinaryInfC{$(\rho_L\wedge\rho_R)(\psi)$}
\DisplayProof

The fact that this replacement works depends on the nature of the other rules present in the proof system---the substitution of $\mu xX.\phi$ with $\psi$ commutes with the introduction rules, and each of the introduction rules is preserved under this substitution. (This is precisely the situation which will become more complicated later.)

\begin{proof}
  We construct $\mathrm{Subst}^{n \in \mu xX. \phi\mapsto \psi(n)}(\sigma)$ by induction on $\sigma$.  We take $\mathrm{Subst}^{n\in\mu xX.\phi\mapsto\psi(n)}(\langle\rangle)$ to be $\mathrm{Read}^{\mathsf{ID}_1^{\infty,+}}_{[n\in\mu xX.\phi]^{n\in\mu xX.\phi,\langle\rangle}}$.

  Suppose that $\mathrm{Subst}^{n\in\mu xX.\phi\mapsto\psi(n)}(\sigma)=\mathrm{Read}^{\mathsf{ID}_1^{\infty,+}}_{[\theta(\mu xX.\phi)]^{n\in\mu xX.\phi,\iota}}$. Consider the rule $\mathcal{R}(\theta(\mu xX.\phi))$; it must be an instance of one of the rules $\wedge \mathrm{I}$, $\vee \mathrm{I}^L$, $\vee\mathrm{I}^R$, $\omega$, $\exists\mathrm{I}$, or $\mathrm{Cl}$. If this rule is anything other than $\mathrm{Cl}$ then there is a corresponding rule $\mathcal{R}(\theta(\psi))$, so we set $\mathrm{Subst}^{n\in\mu xX.\phi\mapsto\psi(n)}(\sigma\top)=\mathcal{R}(\theta(\psi))$ and, for each $\iota'\in|\mathcal{R}(\theta(\psi))|=|\mathcal{R}(\theta(\mu xX.\phi))|$, $\mathrm{Subst}^{n\in\mu xX.\phi\mapsto\psi(n)}(\sigma\top)=\mathrm{Read}^{\mathsf{ID}_1^{\infty,+}}_{[\theta_{\iota'}(\mu xX.\phi)]^{n\in\mu xX.\phi,\iota'}}$.

If $\mathcal{R}$ is Cl$_{k\in \mu xX. \phi}$ then we take $\mathrm{Subst}^{n \in \mu xX. \phi\mapsto \psi(n)}(\sigma\mathcal{R})$ to be

\AxiomC{$\vdots$}
\noLine
\UnaryInfC{$[\Theta_\sigma[Z\mapsto\mu xX.\phi],\phi(k,\mu xX.\phi)]^{n\in\mu xX.\phi,\epsilon\top},\Theta_\sigma[Z\mapsto\psi],\phi(k,\psi) ,\exists x\, \phi(x,\psi)\wedge{\sim}\psi(x)$}
\LeftLabel{\scriptsize{$\mathrm{Read}^{\mathsf{ID}^{\infty,+}_1}_{[\Theta'_{\sigma\top L}]^{n\in\mu xX.\phi, \epsilon\top} }$}}
\UnaryInfC{$[\Theta_\sigma[Z\mapsto\mu xX.\phi]]^{n\in\mu xX.\phi,\epsilon},\Theta_\sigma[Z\mapsto\psi],\phi(k,\psi) ,\exists x\, \phi(x,\psi)\wedge{\sim}\psi(x)$}
\AxiomC{$d_{\psi(k)}$}
\noLine
\UnaryInfC{$\vdots$}
\noLine
\UnaryInfC{$\psi(k),{\sim}\psi(k)$}
\LeftLabel{I$\wedge$}
\BinaryInfC{$[\Theta_\sigma[Z\mapsto\mu xX.\phi]]^{n\in\mu xX.\phi,\epsilon},\Theta_\sigma[Z\mapsto\psi],\exists x\, \phi(x,\psi)\wedge{\sim}\psi(x),\phi(k,\psi)\wedge{\sim}\psi(k)$}
\LeftLabel{I$\exists$}
\UnaryInfC{$[\Theta_\sigma[Z\mapsto\mu xX.\phi]]^{n\in\mu xX.\phi,\epsilon},\Theta_\sigma[Z\mapsto\psi],\exists x\, \phi(x,\psi)\wedge{\sim}\psi(x)$}
\DisplayProof

\bigskip

and set $\Theta_{\sigma\top L}=\Theta_\sigma\cup\{\eta_\top\}$ (where $\eta_\top$ is $\phi(k,\psi)$ by definition). (Note that $\psi(k)$ is, again by definition, contained in $\Theta_\sigma[Z\mapsto\psi]$.)

\end{proof}

\begin{theorem}\label{thm:id1_embedding}
  Let $d$ be a deduction in $\mathsf{ID}_1$ so that $\Gamma(d)$ has free variables contained in $x_1,\ldots,x_n$. There is some $k$ so that, for any numerals $m_1,\ldots,m_n$, there is a proof tree $d^\infty$ in $\mathsf{ID}_{1,<k}^\infty$ with $\Gamma(d^\infty)\subseteq\Gamma(d)[x_i\mapsto m_i]$.
\end{theorem}
\begin{proof}
  We proceed by induction on $d$. If $d$ is $\mathrm{Def}_\phi$ then $\phi[x_i\mapsto m_i]$ is a true quantifier-free formula whose only atomic formulas have the form $\dot{R}t_1\cdots t_n$, and therefore there is a proof of $\phi[x_i\mapsto m_i]$ even in $\mathsf{ID}_{1}^{\infty,+}$.

  If $d$ is $\mathrm{Ax}_\eta$ then either $\eta[x_i\mapsto m_i]$ has the form $\dot{R}t_1\cdots t_n$, in which case $\mathrm{True}_{\eta[x_i\mapsto m_i]}$ suffices, or $\eta[x_i\mapsto m_i]$ has the form $n\in\mu xX.\phi$, in which case the deduction $d_{n\in\mu xX.\phi}$ is given above.  When the final inference rule of $d$ is $\wedge$I, $\vee$I, $\exists$I, $\mathrm{Cl}$, or $\mathrm{Cut}$, the claim follows immediately from the inductive hypothesis. (In the case of $\exists$I and $\mathrm{Cut}$, the premises may have additional free variables; we substitute these free variables with $0$'s to apply the inductive hypothesis.)

  When the final inference rule of $d$ is $\forall^y_{\forall x\,\phi}$I, we have a premise $d'$ concluding $\Gamma(d),\phi(y)$ where $y$ is not free in $\Gamma(d)$. For each $n$, the inductive hypothesis gives a deduction of $(\Gamma(d),\phi(n))[x_i\mapsto m_i]$, and we may apply an $\omega$ rule to obtain the desired deduction.

  When $d$ is $\mathrm{Ind}^t_{\forall x\,\phi}$, we let $n$ be $\overline{t[x_i\mapsto m_i]}$ and proceed by induction on $n$. For notational convenience, let $\phi'$ be $\phi[x_i\mapsto m_i]$. When $n=0$, we may apply the previous lemma to obtain a deduction $d^\infty_0$ with $\Gamma(d^\infty_0)\subseteq \phi'(0),{\sim}\phi(0)$. So suppose we have a deduction $d^\infty_n$ with $\Gamma(d^\infty_n)\subseteq {\sim}\phi'(0), \exists x\,\phi'(x)\wedge{\sim}\phi'(Sx)\, \phi'(n)$. Then we may take $d^\infty_{n+1}$ to be

  \AxiomC{$d^\infty_n$}
  \noLine
  \UnaryInfC{$\vdots$}
  \noLine
  \UnaryInfC{$ {\sim}\phi'(0), \exists x\,\phi'(x)\wedge{\sim}\phi'(Sx), \phi'(n)$}
  \AxiomC{$d_{\phi'(n+1)}$}
  \noLine
  \UnaryInfC{$\vdots$}
  \noLine
  \UnaryInfC{${\sim}\phi'(n+1),\phi'(n+1)$}
  \BinaryInfC{$ {\sim}\phi'(0), \exists x\,\phi'(x)\wedge{\sim}\phi'(Sx), \phi'(n)\wedge{\sim}\phi'(n+1), \phi'(n+1)$}
  \UnaryInfC{$ {\sim}\phi'(0), \exists x\,\phi'(x)\wedge{\sim}\phi'(Sx),\phi'(n+1)$}
  \DisplayProof

  The novel case is when $d$ is $\mathrm{Ind}^t_{\mu xX\, \phi,\psi(y)}$. Again, let us write $\phi'$ for $\phi[x_i\mapsto m_i]$, $t'$ for $t[x_i\mapsto m_i]$, and $\psi'$ for $\psi[x_i\mapsto m_i]$. In this case we may use an $\Omega^\flat$ inference rule applied to the function $\mathrm{Subst}^{t'\in \mu xX.\phi'\mapsto \psi'}$ from the previous lemma.
\end{proof}

\subsection{Cut Elimination}

We next show that it is possible to eliminate cuts from deductions in $\mathsf{ID}_1^\infty$. The proof is essentially standard; we need only check that the new rules do not interfere with the usual arguments. 

\begin{lemma}[$\bot$ Inversion]
  Let $\eta$ be an atomic formula of the form $\dot{R}m_1\cdots m_n$ and let $\phi\in\{\eta,\neg\eta\}$ be false. If $d$ is a $\mathsf{ID}^\infty_{1,0}$ proof tree then $\phi\not\in\Gamma(d)$.
\end{lemma}
\begin{proof}
Immediate, since no rule in $\mathsf{ID}^\infty_{1,0}$ can introduce a false literal of this form.
\end{proof}

\begin{lemma}[$\wedge$ Inversion]\label{thm:and_inversion}
  For any formulas $\phi_L,\phi_R$ with rank $< n$ and $i\in\{L,R\}$ and any $\mathsf{ID}^\infty_{1,<n}$ proof tree $d$, there is a $\mathsf{ID}^\infty_{1,<n}$ proof tree $\mathrm{Inverse}^\wedge_{\phi_L,\phi_R, i}(d)$ with $\Gamma(\mathrm{Inverse}^\wedge_{\phi_L,\phi_R, i}(d))\subseteq(\Gamma(d)\setminus\{\phi_L\wedge\phi_R\})\cup\{\phi_i\}$.%
\end{lemma}
\begin{proof}
  We define $\mathrm{Inverse}^\wedge_{\phi_L,\phi_R, i}(d)(\sigma)$ by induction on $\sigma$, and simultaneously define a function $\pi:\dom(\mathrm{Inverse}^\wedge_{\phi_L,\phi_R, i}(d))\rightarrow\dom(d)$. %

  We set $\pi(\langle\rangle)=\langle\rangle$.  Let $\pi(\sigma)$ be defined. If $\phi_L\wedge\phi_R\not\in\Delta(d(\pi(\sigma))$, we set $\mathrm{Inverse}^\wedge_{\phi_L,\phi_R, i}(d)(\sigma)=d(\pi(\sigma))$ and, for each $\iota\in|\mathrm{Inverse}^\wedge_{\phi_L,\phi_R, i}(d)(\sigma)|$, $\pi(\sigma\iota)=\pi(\sigma)\iota)$.

  Suppose $\phi_L\wedge\phi_R\in\Delta(d(\pi(\sigma)))$. If $d(\pi(\sigma))$ is $\wedge\mathrm{I}_{\phi_L\wedge\phi_R}$ then we set $\mathrm{Inverse}^\wedge_{\phi_L,\phi_R, i}(d)(\sigma)=\mathrm{Rep}$ and $\pi(\sigma\top)=\pi(\sigma)i$.

  The only other rule in $\mathsf{ID}^\infty_{1,<n}$ with $\phi_L\wedge\phi_R\in\Delta(d(\pi(\sigma))$ is $\mathrm{Read}^{\mathsf{ID}^{\infty,+}_1}_{[\Theta]^{n\in\mu xX.\phi,\epsilon}}$ with $\phi_L\wedge\phi_R\in\Delta(\epsilon)\setminus\Theta$. We set $\mathrm{Inverse}^\wedge_{\phi_L,\phi_R, i}(d)(\sigma)=\mathrm{Read}^{\mathsf{ID}^{\infty,+}_1}_{[\Theta\cup\{\phi_L\wedge\phi_R\}]^{n\in\mu xX.\phi,\epsilon}}$ and, for all $\mathcal{R}$, $\pi(\sigma\mathcal{R})=\pi(\sigma)\mathcal{R}$.
\end{proof}

\begin{lemma}[$\forall$ Inversion]\label{thm:all_inversion}
  For any formula $\phi(y)$ with rank $< n$ whose only free variable is $y$, any $m\in\mathbb{N}$, and any $\mathsf{ID}^\infty_{1,<n}$ proof tree $d$, there is a $\mathsf{ID}^\infty_{1,<n}$ proof tree $\mathrm{Inverse}^\forall_{\phi(y),m}(d)$ with $\Gamma(\mathrm{Inverse}^\forall_{\phi(y),m} (d))\subseteq(\Gamma(d)\setminus \{\forall x\,\phi(x)\})\cup\{\phi(m)\}$.%
\end{lemma}
\begin{proof}
  We define $\mathrm{Inverse}^\forall_{\phi(y),m} (d)(\sigma)$ by induction on $\sigma$, and simultaneously define a function $\pi:\dom(\mathrm{Inverse}^\forall_{\phi(y),m} (d))\rightarrow\dom(d)$. %

  We set $\pi(\langle\rangle)=\langle\rangle$.  Let $\pi(\sigma)$ be defined. If $\forall x\,\phi(x)\not\in\Delta(d(\pi(\sigma)))$ then we set $\mathrm{Inverse}^\forall_{\phi(y),m} (d) (\sigma)=d(\pi(\sigma))$ and, for each $\iota\in|\mathrm{Inverse}^\forall_{\phi(y),m} (d) (\sigma)|$, $\pi(\sigma\iota)=\pi(\sigma)\iota)$.

  If $d(\pi(\sigma))$ is $\omega_{\forall x\,\phi}$ then we set $\mathrm{Inverse}^\forall_{\phi(y),m} (d) (\sigma)=\mathrm{Rep}$ and $\pi(\sigma\top)=\pi(\sigma)m$.

  The only other rule with $\forall x\,\phi(x)\in\Delta(d(\pi(\sigma)))$ is $\mathrm{Read}^{\mathsf{ID}^{\infty,+}_1}_{[\Theta]^{n\in\mu xX.\phi,\epsilon}}$ with $\forall x\,\phi(x)\in\Delta(\epsilon)\setminus\Theta$. We set $\mathrm{Inverse}^\forall_{\phi(y),m} (d)(\sigma)=\mathrm{Read}^{\mathsf{ID}^{\infty,+}_1}_{[\Theta\cup\{\forall x\,\phi\}]^{n\in\mu xX.\phi,\epsilon}}$ and, for all $\mathcal{R}$, $\pi(\sigma\mathcal{R})=\pi(\sigma)\mathcal{R}$.
\end{proof}

\begin{lemma}[$\vee$ Elimination]\label{thm:or_elimination}
  For any formulas $\phi_L,\phi_R$ with rank $< n$ and any $\mathsf{ID}^\infty_{1,<n}$ proof trees $d_\wedge$ and $d_\vee$, there is a $\mathsf{ID}^\infty_{1,<n}$ proof tree $\mathrm{Elim}^\vee_{\phi_L,\phi_R}(d_\vee,d_\wedge)$ with $\Gamma(\mathrm{Elim}^\vee_{\phi_L,\phi_R}(d_\vee,d_\wedge))\subseteq (\Gamma(d_\wedge)\setminus\{\phi_L\wedge\phi_R\})\cup(\Gamma(d_\vee)\setminus\{{\sim}\phi_L\vee{\sim}\phi_R\})$.%
\end{lemma}
\begin{proof}

  We define a partial function $\pi:\dom(\mathrm{Elim}^\vee_{\phi_L,\phi_R}(d_\vee,d_\wedge))\rightarrow\dom(d_\vee)$.  We define $\pi(\langle\rangle)=\langle\rangle$. %
  Suppose we have defined $\pi(\sigma)$. If ${\sim}\phi_L\vee{\sim}\phi_R\not\in\Delta(d_\vee(\pi(\sigma)))$ then we set $\mathrm{Elim}^\vee_{\phi_L,\phi_R}(d_\vee,d_\wedge)(\sigma)=d_\vee(\pi(\sigma))$ and, for all $\iota\in|\mathrm{Elim}^\vee_{\phi_L,\phi_R}(d_\vee,d_\wedge)(\sigma)|$, $\pi(\sigma\iota)=\pi(\sigma)\iota$.
  
  If $d_\vee(\pi(\sigma))$ is $\vee\mathrm{I}^i_{{\sim}\phi_L\vee{\sim}\phi_R}$, we set $\mathrm{Elim}^\vee_{\phi_L,\phi_R}(d_\vee,d_\wedge)(\sigma)=\mathrm{Cut}_{\phi_i}$. We set $\pi(\sigma R)=\pi(\sigma)$, and for all $\upsilon\in\dom(\mathrm{Inverse}^\wedge_{\phi_L,\phi_R,i}(d_\wedge))$, we set $\mathrm{Elim}^\vee_{\phi_L,\phi_R}(d_\vee,d_\wedge)(\sigma L\upsilon)=\mathrm{Inverse}^\wedge_{\phi_L,\phi_R,i}(d_\wedge)(\upsilon)$.%

  If $d_\vee(\pi(\sigma))$ is $\mathrm{Read}^{\mathsf{ID}^{\infty,+}_1}_{[\Theta]^{n\in\mu xX.\phi,\epsilon}}$ with ${\sim}\phi_L\vee{\sim}\phi_R\in\Delta(\epsilon)\setminus\Theta$ then we set $\mathrm{Elim}^\vee_{\phi_L,\phi_R}(d_\vee,d_\wedge)(\sigma)=\mathrm{Read}^{\mathsf{ID}^{\infty,+}_1}_{[\Theta\cup\{{\sim}\phi_L\vee{\sim}\phi_R\}]^{n\in\mu xX.\phi,\epsilon}}$ and, for all $\mathcal{R}$, $\pi(\sigma\mathcal{R})=\pi(\sigma)\mathcal{R}$.
\end{proof}

\begin{lemma}[$\exists$ Elimination]\label{thm:exists_elimination}
  For any formula $\phi(y)$ with rank $< n$ whose only free variable is $y$ and any $\mathsf{ID}^\infty_{1,<n}$ proof trees $d_\forall$ and $d_\exists$, there is a $\mathsf{ID}^\infty_{1,<n}$ proof tree $\mathrm{Elim}^\exists_{\phi(y)}(d_\exists,d_\forall)$ with $\Gamma(\mathrm{Elim}^\exists_{\phi(y)}(d_\exists,d_\forall))\subseteq (\Gamma(d_\forall)\setminus\{\forall y\,\phi(y)\})\cup(\Gamma(d_\exists)\setminus\{\exists y\,{\sim}\phi(y))$.%
\end{lemma}
\begin{proof}

  We define $\pi(\langle\rangle)=\langle\rangle$.  %
  Suppose we have defined $\pi(\sigma)$. If $\exists\mathrm{I}^t_{\exists y\,{\sim}\phi}$ is not in $\Delta(d_\exists(\pi(\sigma)))$ then we set $\mathrm{Elim}^\exists_{\phi(y)}(d_\exists,d_\forall)(\sigma)=d_\exists(\pi(\sigma))$ and, for all $\iota\in|\mathrm{Elim}^\exists_{\phi(y)}(d_\exists,d_\forall)(\sigma)|$, $\pi(\sigma\iota)=\pi(\sigma)\iota$.

  If $d_\exists(\pi(\sigma))$ is $\exists\mathrm{I}^t_{\exists y\,{\sim}\phi}$ then, since $t$ is closed, $t=m$ for some numeral $m$. We set $\mathrm{Elim}^\exists_{\phi(y)}(d_\exists,d_\forall)(\sigma)=\mathrm{Cut}_{\phi(m)}$. We set $\pi(\sigma R)=\pi(\sigma)\top$ and, for all $\upsilon\in\dom(\mathrm{Inverse}^\forall_{\phi(y),m}(d_\forall))$, we set $\mathrm{Elim}^\exists_{\phi(y)}(d_\exists,d_\forall)(\sigma L\upsilon)=\mathrm{Inverse}^\forall_{\phi(y),m}(d_\forall)(\upsilon)$.%

  If $d_\exists(\pi(\sigma))$ is $\mathrm{Read}^{\mathsf{ID}_1^{\infty,+}}_{[\Theta]^{n\in\mu xX.\phi,\epsilon}}$ with $\exists y\,{\sim}\phi\in\Delta(\epsilon)\setminus\Theta$ then we set $\mathrm{Elim}^\exists_{\phi(y)}(d_\exists,d_\forall)(\sigma)=\mathrm{Read}^{\mathsf{ID}_1^{\infty,+}}_{[\Theta\cup\{\exists y\,{\sim}\phi\}]^{n\in\mu xX.\phi,\epsilon}}$ and, for all $\mathcal{R}$, $\pi(\sigma\mathcal{R})=\pi(\sigma)\mathcal{R}$.
\end{proof}

\begin{lemma}[$\mu$ Elimination]\label{thm:mu_elimination}
For any atomic formula of the form $n\in\mu xX.\phi$ and any $\mathsf{ID}_{1,<0}^{\infty}$ proof trees $d_\mu$ and $d_\neg$, there is a $\mathsf{ID}_{1,<0}^\infty$ proof tree $\mathrm{Elim}^\mu_{n\in\mu xX.\phi}(d_\mu,d_\neg)$ with $\Gamma(\mathrm{Elim}^\mu_{n\in\mu xX.\phi}(d_\mu,d_\neg))\subseteq(\Gamma(d_\mu)\setminus\{n\in \mu xX.\phi\})\cup(\Gamma(d_\neg)\setminus\{n\not\in\mu xX.\phi\})$.%
\end{lemma}
\begin{proof}
  We define $\pi(\langle\rangle)=\langle\rangle$. %

  Suppose we have defined $\pi(\sigma)$. If $n\not\in\mu xX.\phi\not\in\Delta(d_\neg(\pi(\sigma)))$ then we set $\mathrm{Elim}^\mu_{n\in\mu xX.\phi}(d_\mu,d_\neg) (\sigma)=d_\neg(\pi(\sigma))$ and, for all $\iota\in|\mathrm{Elim}^\mu_{n\in\mu xX.\phi}(d_\mu,d_\neg) (\sigma)|$, $\pi(\sigma\iota)=\pi(\sigma)\iota$.

  If $n\not\in\mu xX.\phi\in\Delta(d_\neg(\pi(\sigma)))$ then if $d_\neg(\pi(\sigma))$ is $\Omega^\flat_{n\not\in\mu xX.\phi}$, we set $\mathrm{Elim}^\mu_{n\in\mu xX.\phi}(d_\mu,d_\neg) (\sigma)=\mathrm{Cut}\Omega^\flat_{n\not\in\mu xX.\phi}$, $\pi(\sigma R)=\pi(\sigma)\top$, and $\mathrm{Elim}^\mu_{n\in\mu xX.\phi}(d_\mu,d_\neg)(\sigma L \upsilon)=d_\mu(\upsilon)$.%

  The final case is that $d_\neg(\pi(\sigma))$ is $\mathrm{Read}^{\mathsf{ID}_1^{\infty,+}}_{[\Theta]^{m\in yY.\phi,\epsilon}}$ with $n\not\in\mu xX.\phi\in\Delta(\epsilon)\setminus\Theta$. In this case we set $\mathrm{Elim}^\mu_{n\in\mu xX.\phi}(d_\mu,d_\neg) (\sigma)=\mathrm{Read}^{\mathsf{ID}_1^{\infty,+}}_{[\Theta\cup\{n\not\in\mu xX.\phi\}]^{m\in yY.\phi,\epsilon}}$ and, for all $\mathcal{R}$, $\pi(\sigma\mathcal{R})=\pi(\sigma)\mathcal{R}$.
\end{proof}

\begin{theorem}
If $d$ is a proof tree in $\mathsf{ID}_{1,<n+1}^{\infty}$ then there is a proof tree $\mathrm{Reduce}_n(d)$ in $\mathsf{ID}_{1,<n}^\infty$ with $\Gamma(\mathrm{Reduce}_n(d))\subseteq\Gamma(d)$.%
\end{theorem}
\begin{proof}
  For every $\tau\in\dom(d)$, we need to define a proof tree $\mathrm{Reduce}_n(d,\tau)$, which represents beginning the cut reduction process above the position $\tau$.%

  As always, we define $\mathrm{Reduce}_n(d,\tau)(\sigma)$ by induction on $\sigma$ while simultaneously defining a partial function $\pi_\tau:\dom(\mathrm{Reduce}_n(d,\tau))\rightarrow\dom(d)$. Initially, we set $\pi_\tau(\langle\rangle)=\tau$.%

If $d(\pi_\tau(\sigma))$ is any rule other than a $\mathrm{Cut}_\phi$ with $rk(\phi)=n$, we set $\mathrm{Reduce}_n(d,\tau)(\sigma)=d(\pi(\sigma))$ and $\pi_\tau(\sigma\iota)=\pi_\tau(\sigma)\iota$ for all $\iota\in|d'(\sigma)|$.

  So suppose $d(\pi_\tau(\sigma))$ is $\mathrm{Cut}_\phi$ with $rk(\phi)=n$. We set $\mathrm{Reduce}_n(d,\tau)(\sigma)=\mathrm{Rep}$ and consider cases on $\phi$. Without loss of generality, we assume $\phi$ is either an atomic formula, a conjunction, or a universal formula. (Otherwise $\phi$ is the negation of such a formula and we apply the same arguments after swapping the input branches.)

  If $\phi$ is an atomic formula $\dot{R}t_1\cdots t_n$ then $\phi$ is either true or false. If $\phi$ is false, we define $\mathrm{Reduce}_n(d,\tau)(\sigma \top\upsilon)=\mathrm{Reduce}_n(d, \pi_\tau(\sigma)L)(\upsilon)$ for all $\upsilon$. If $\phi$ is true, so $\neg \phi$ is false, we define $\mathrm{Reduce}_n(d,\tau)(\sigma\top\upsilon)=\mathrm{Reduce}_n(d, \pi_\tau(\sigma)R)(\upsilon)$.%

  If $\phi$ is a conjunction $\phi_L\wedge\phi_R$ (respectively, a universal $\forall x\,\phi(x)$ or atomic formula $t\in\mu xX.\phi$), we take $\mathrm{Reduce}_n(d,\tau)(\sigma\top\upsilon)$ to be
  \[\mathrm{Elim}^\vee_{\phi_L,\phi_R}(\mathrm{Reduce}_n(d,\tau \pi_\tau(\sigma)R),\mathrm{Reduce}_n(d, \pi_\tau(\sigma)L))(\upsilon)\]
  or
  \[\mathrm{Elim}^\exists_{\phi(x)}(\mathrm{Reduce}_n(d,\pi_\tau(\sigma)R),\mathrm{Reduce}_n(d, \pi_\tau(\sigma)L))(\upsilon)\]
  or
  \[\mathrm{Elim}^\mu_{t\in\mu xX.\phi}(\mathrm{Reduce}_n(d, \pi_\tau(\sigma)R),\mathrm{Reduce}_n(d,\pi_\tau(\sigma) L))(\upsilon).\]
\end{proof}

\subsection{Collapsing}

We define when a $\mu$-expression appears positively or negatively in a formula.
\begin{definition}
  We say $\mu xX.\phi$ \emph{appears positvely} (resp. negatively) in a $\mathcal{L}_\mu$ formula $\psi$ if there are $\psi', Z$ so that $\psi=\psi'[Z\mapsto \mu xX.\phi]$ is a permissible substitution, $Z\in FV(\psi')$ and $Z\in pos(\psi')$ (resp. $Z\in neg(\psi')$).

  We say $\mu xX.\phi$ appears positively (resp. negatively) in a set $\Gamma$ if there is some $\psi$ in $\Gamma$ so that $\mu xX.\phi$ appears positively (resp. negatively) in $\psi$.

  We say $\mu xX.\phi$ appears positively (resp. negatively) in $[\Theta]^{n\in\mu yY.\psi,\epsilon}$ if $\mu xX.\phi$ appears negatively (resp. positively) in $\Theta$.
\end{definition}
Note that $[\Theta]$ ``swaps the polarity'' of occurences---something appears positively in $[\Theta]$ if it appears \emph{negatively} in $\Theta$.

\begin{theorem}\label{thm:id1_collapsing}
  Suppose $d$ is a proof tree in $\mathsf{ID}_{1,<0}^\infty$ so that no $n\in\mu xX.\phi$ appears negatively in $\Gamma(d)$. Then there is a proof tree $\mathcal{D}(d)$ in $\mathsf{ID}^{\infty,+}_1$ with $\Gamma(\mathcal{D}(d))\subseteq\Gamma(d)$.
\end{theorem}
\begin{proof}
  We need to define $\mathcal{D} (d)(\sigma)$ by induction on $|\sigma|$ for all $d$ satisfying the assumptions simultaneously, and we simultaneously define partial functions $\pi_d:\dom(\mathcal{D}(d))\rightarrow \dom(d)$. We also need to ensure that, for all $\sigma\in\dom(\mathcal{D}(d))$ for which $\pi_d(\sigma)$ is defined, $\Gamma(d,\pi_d(\sigma))$ does not contain $n\not\in\mu xX.\phi$ or $[\Sigma]^{n\in\mu xX.\phi,\epsilon}$.  As usual, we set $\pi_d(\langle\rangle)=\langle\rangle$.

  Suppose we have defined $\pi_{d}(\sigma)$. Observe that $d(\pi_d(\sigma))$ cannot be $\mathrm{Read}$ or $\Omega^\flat$. So if $d(\pi_d(\sigma))$ is any rule other than $\mathrm{Cut}\Omega^\flat$, it must be a rule in $\mathsf{ID}^{\infty,+}_1$, so we can set $\mathcal{D} (d)(\sigma)=d(\pi_d(\sigma))$ and, for $\iota\in |\mathcal{D} (d)(\sigma)|$, $\pi_d(\sigma\iota)=\pi_d(\sigma)\iota$.

  So suppose $d(\pi_d(\sigma))=\mathrm{Cut}\Omega^\flat_{n\in\mu xX.\phi}$. We will set $\mathcal{D} (d)(\sigma)=\mathrm{Rep}$ and, for all $\upsilon$, $\mathcal{D} (d)(\sigma\top \upsilon)=\mathcal{D} (\bar d_\bot(\mathcal{D} (d_\top)))(\upsilon)$.   Note that this is well-defined using the inductive hypothesis since $|\upsilon|<|\sigma\top\upsilon|$ and $\mathcal{D} (\bar d_\bot(\mathcal{D} (d_\top)))(\upsilon)$ depends only on $\bar d_\bot(\mathcal{D} (d_\top))(\upsilon')$ with $|\upsilon'|\leq|\upsilon|$, which in turn depends on $\mathcal{D} (d_\top)(\upsilon'')$ for $\upsilon''$ with $|\upsilon''|\leq|\upsilon'|\leq|\upsilon|<|\sigma\top\upsilon|$.
\end{proof}

\subsection{Ordinal Terms}\label{sec:id1_ordinal_terms}

Putting together the results in the previous subsections, we may take a proof in $\mathsf{ID}_1$ whose conclusion does not contain negative occurences of $\mu$ formulas and obtain a proof tree $\mathcal{D} (\mathrm{Reduce}_{n}(\cdots\mathrm{Reduce}_0(d^\infty)\cdots ))$ in $\mathsf{ID}_1^{\infty,+}$ with the same conclusion. In order for this to be of use, we need to ensure that the proof tree we get at the end of this process is actually well-founded.

This is complicated by the fact that $d^\infty$ is typically \emph{not} well-founded. Nevertheless, it has an ordinal bound of a more general sort, and we can extract an ordinal notation which will let us demonstrate this by careful inspection of the proof.

The distinctive feature of the approach to ordinal notations here is the need to give suitable ordinal bounds for $\mathcal{D} (d)$---in particular, when $d(\langle\rangle)=\mathrm{Cut}\Omega^\flat$, we need the ordinal $\mathcal{D} (d)$ to bound the ordinal of $\mathcal{D} (\bar d_\bot(\mathcal{D} (d_\top)))$. This, in turn, requires us to make sense of the ordinal $\bar d_\bot(\mathcal{D} (d_\top))$---that is, to think of $d_\bot$, as being assigned, not an ordinal, but a function on ordinals.

To handle this, we will assign ``ordinal terms'' to nodes in our proof trees. Ordinal terms will be allowed to contain variables---specifically, if $[\Sigma]^{n\in\mu xX.\phi,\epsilon}\in\Gamma(d,\sigma)$ then the ordinal bound at $\sigma$ will be allowed to contain a variable $v_{n\in\mu xX.\phi,\epsilon}$. The symbol $\Omega_1$ (written as just $\Omega$ in this context) acts as a sort of ``meta-variable'' that bounds all other variables.

\begin{definition}\label{ref:ordinal_terms_id1}
We define several notions simultaneously---we will define the ordinal terms, $\mathrm{OT}$, together with some special classes of ordinal terms $\mathbb{V}\subseteq\mathbb{SC}\subseteq\mathbb{H}\subseteq\mathrm{OT}$, and the notion of a free variable.

The \emph{ordinal terms}, $\mathrm{OT}$, are given as follows:
\begin{enumerate}
\item if $\{\alpha_0,\ldots,\alpha_{n-1}\}$ is a finite multi-set of ordinal terms in $\mathbb{H}$ and $n\neq 1$ then $\#\{\alpha_0,\ldots,\alpha_{n-1}\}$ is an ordinal term,
\item $\Omega$ is an ordinal term in $\mathbb{V}$,
\item for any $n\in\mu xX.\phi,\epsilon$, $v_{n\in\mu xX.\phi,\epsilon}$ are ordinal terms in $\mathbb{V}$,
\item if $\alpha$ is an ordinal term then $\omega^\alpha$ is an ordinal term in $\mathbb{H}$,
\item if $\alpha$ is an ordinal term with $FV(\alpha)\subseteq\{\Omega\}$ then $\vartheta\alpha$ is an ordinal term in $\mathbb{SC}$.
\end{enumerate}

We define the free variables in a term:
\begin{itemize}
\item $FV(\#\{\alpha_i\})=\bigcup_iFV(\alpha_i)$,
\item $FV(\Omega)=\{\Omega\}$,
\item $FV(v_{n\in\mu xX.\phi,\epsilon})=\{v_{n\in\mu xX.\phi,\epsilon}\}$,
\item $FV(\omega^\alpha)=FV(\alpha)$,
\item $FV(\vartheta\alpha)=FV(\alpha)\setminus\{\Omega\}$.
\end{itemize}
We say $\alpha$ is closed if $FV(\alpha)=\emptyset$.

We define the ordering on ordinal terms by:
\begin{itemize}
\item $\#\{\alpha_0,\ldots,\alpha_{m-1}\}<\#\{\beta_0,\ldots,\beta_{n-1}\}$ if there is a $\gamma\in \{\beta_0,\ldots,\beta_{n-1}\}\setminus \{\alpha_0,\ldots,\alpha_{m-1}\}$ so that, for all $\delta\in \{\alpha_0,\ldots,\alpha_{m-1}\}\setminus \{\beta_0,\ldots,\beta_{n-1}\}$, $\delta<\gamma$,
\item when $\beta\in\mathbb{H}$, $\#\{\alpha_0,\ldots,\alpha_{m-1}\}<\beta$ if, for each $i$, $\alpha_i<\beta$,
\item when $\beta\in\mathbb{H}$, $\beta<\#\{\alpha_0,\ldots,\alpha_{m-1}\}$ if there is some $i$ so that $\beta\leq\alpha_i$,
\item $\omega^\alpha<\omega^\beta$ if $\alpha<\beta$,
\item if $\gamma\in\mathbb{SC}$ then $\omega^\alpha<\gamma$ if $\alpha<\gamma$,
\item if $\gamma\in\mathbb{SC}$ then $\gamma<\omega^\alpha$ if $\gamma\leq \alpha$,
\item $\vartheta\alpha<\vartheta\beta$ if either:
  \begin{itemize}
  \item $\alpha<\beta$ and for all $\gamma\in K\alpha$, $\gamma<\vartheta\beta$,
  \item $\beta<\alpha$ and there is a $\gamma\in K\alpha$ so that $\vartheta\beta\leq\gamma$.
  \end{itemize}
\item $\vartheta\alpha<v_{m\in\mu xX.\phi,\epsilon}$,
\item $\vartheta\alpha<\Omega$,
\item $v_{m\in\mu xX.\phi,\epsilon}<v_{n\in\mu yY.\psi,\epsilon'}$ if and only if $m\in\mu xX.\phi$ and $n\in\mu yY.\psi$ are identical and $\epsilon'\sqsubset\epsilon$,
\item $v_{m\in\mu xX.\phi,\epsilon}<\Omega$.
\end{itemize}
   
We define $K\alpha$ by:
\begin{itemize}
\item $K\#\{\alpha_0,\ldots,\alpha_{n-1}\}=\bigcup_i K\alpha_i$,
\item $K\Omega=\emptyset$,
\item $Kv_{n\in\mu xX.\phi,\epsilon}=\emptyset$,
\item $K\omega^\alpha=K\alpha$,
\item $K\vartheta\alpha=\{\vartheta\alpha\}$.
\end{itemize}

We say $\alpha\ll_\gamma\beta$ if $\alpha<\beta$ and, for all $\delta\in K\alpha$, we have $\delta<\max\{\vartheta\beta,\gamma\}$.

\end{definition}

We adopt the abbreviation $0$ for $\#\emptyset$, and give $\alpha\#\beta$ the natural interpretation---if $\alpha,\beta\in\mathbb{H}$ then $\alpha\#\beta$ is $\{\alpha,\beta\}$; if $\alpha=\#\{\alpha_0,\ldots,\alpha_{n-1}\}$ then $\alpha\#\beta$ is $\#\{\alpha_0,\ldots,\alpha_{n-1},\beta\}$, and similarly if $\beta\not\in\mathbb{H}$. When $k$ is finite, we write $\alpha\cdot k$ for the $\#$ sum of $k$ copies of $\alpha$.

\begin{definition}
  When $\gamma$ is closed in $\mathbb{SC}$, we define $\vartheta_\gamma(\alpha)=\vartheta(\omega^{\Omega\#\alpha}\#\gamma)$.
\end{definition}

\begin{lemma}
  If $\alpha\ll_\gamma\beta$ with $FV(\alpha)\cup FV(\beta)\subseteq\{\Omega\}$ then we have $\vartheta_\gamma(\alpha)<\vartheta_\gamma(\beta)$.
\end{lemma}
\begin{proof}
  We have $\omega^{\Omega\#\alpha}\#\gamma<\omega^{\Omega\#\beta}\#\gamma$, so consider some $\vartheta\zeta\in K(\omega^{\Omega\#\alpha}\#\gamma)=K\gamma\cup K\alpha$. We must show that $\vartheta\zeta<\vartheta_\gamma(\beta)$.

  If $\vartheta\zeta\in K\gamma$ then we have $\vartheta\zeta\in K(\omega^{\Omega\#\beta}\#\gamma)$, and therefore $\vartheta\zeta<\vartheta_\gamma(\beta)$. So we may assume $\vartheta\zeta\in K\alpha$ with $\gamma<\vartheta\zeta$ (since, similarly, $\gamma<\vartheta_\gamma(\beta)$). We have $\vartheta\zeta<\vartheta\beta<\vartheta(\omega^{\Omega\#\beta}\#\gamma)$, as needed.
\end{proof}

We define substitution of free variables other than $\Omega$ in the usual way. We write $\alpha[v\mapsto \beta]$ for replacing the variable $v$ with $\beta$, and when $f$ is a function from variables to ordinal terms, we simply write $\alpha[f]$ to indicate replacing each $v\in\dom(f)$ with $f(x)$. We always assume that such substitutions are order-preserving.

\begin{lemma}
  If $\alpha\ll_\gamma\beta$ and $f$ is a substitution whose range is $\leq\delta$ where $\delta$ is closed then $\alpha[f]\ll_{\max\{\gamma,\delta\}}\beta[f]$.
\end{lemma}
\begin{proof}
  First we show that $\alpha[f]<\beta[f]$ by induction on $\alpha$ and $\beta$. Most cases follow immediately from either the inductive hypothesis or the fact that $f$ is order-preserving, as well as the fact that we don't make further substitutions inside a $\vartheta\alpha'$ subterm.

Now consider some $\theta\in K\alpha[f]$; we have $\theta\in K\alpha\cup\rng(f)$, so $\theta<\max\{\vartheta\beta,\gamma,\delta\}$.
\end{proof}

\begin{lemma}\label{thm:id1_D_bound}
  Suppose $\alpha\ll_\gamma\beta$ with $FV(\alpha)\cup FV(\beta)\subseteq\{\Omega\}$ and suppose $\delta\ll_\gamma\beta$ with $FV(\delta)\subseteq\{\Omega,v\}$. Then $\vartheta_{\vartheta_\gamma(\alpha)}(\delta[v\mapsto \vartheta_\gamma(\alpha)])<\vartheta_\gamma(\beta)$.
\end{lemma}
\begin{proof}
  Since $\delta[v\mapsto \vartheta_\gamma(\alpha)]<\beta$ by the previous lemma, we have $\omega^{\Omega\#\delta[v\mapsto \vartheta_\gamma(\alpha)]}\#\vartheta_{\gamma}(\alpha)<\omega^{\Omega\#\beta}\#\gamma$ as needed.

  We also have $\vartheta_\gamma(\alpha)<\vartheta_\gamma(\beta)$. If $\zeta\in K(\omega^{\Omega\#\delta[v\mapsto \vartheta_\gamma(\alpha)]}\#\vartheta_\gamma(\alpha))$ then either $\zeta=\vartheta_\gamma(\alpha)<\vartheta_\gamma(\beta)$ or $\zeta\in K(\delta[v\mapsto \vartheta_\gamma(\alpha)])$, so $\zeta\in K\delta\cup\{\vartheta_\gamma(\alpha)\}$ (because $v$ cannot appear in an element of $K\delta$), so again $\zeta<\vartheta_\gamma(\beta)$.
\end{proof}

\begin{definition}
  
When $d$ is a proof tree in $\mathsf{ID}_1^{\infty}$, an \emph{ordinal bound above $\gamma$} on $d$ is a function $o^d:\dom(d)\rightarrow \mathrm{OT}$ such that
  \begin{itemize}
  \item for all $\sigma\in\dom(d)$, if $v_{n\in\mu xX.\phi,\epsilon}\in FV(o^d(\sigma))$ then $[\Theta]^{n\in\mu xX.\phi,\epsilon}\in\Gamma(d,\sigma)$ for some $\Theta$,
  \item if $\sigma\sqsubset\tau$ then $o^d(\tau)\ll_\gamma o^d(\sigma)$.
  \end{itemize}

  We say $d$ \emph{is bounded by $\alpha$ above $\gamma$} if there is an ordinal bound above $\gamma$, $o^d$, with $o^d(\langle\rangle)\lleq_\gamma\alpha$.%

\end{definition}
Note that when $d$ is bounded by $\alpha$ with $\alpha$ closed, the range of $o^d$ must be closed terms, and $\ll$ is equivalent to $<$. In this case we may simply say $d$ is bounded by $\alpha$.

\begin{lemma}
  Let $F$ be bounded by $\alpha$ above $\gamma$ and let $d$ be bounded by a closed ordinal term $\beta$ above $\gamma$. Then $\bar F(d)$ is bounded by $\alpha[v_{n\in\mu xX.\phi,\langle\rangle}\mapsto\beta]$ above $\max\{\gamma,\beta\}$.

\end{lemma}
\begin{proof}
  For each $\sigma$, we define $o^{\bar F(d)}(\sigma)=o^{F}(h(\sigma))[v_{n\in\mu xX.\phi,\epsilon}\mapsto o^d(\epsilon)]$.

  Consider any $\sigma\sqsubset\tau$. We have $o^{\bar F(d)}(\sigma)=o^F(h(\sigma)) [v_{n\in\mu xX.\phi,\epsilon}\mapsto o^d(\epsilon)]\ll_{\max\{\gamma,\beta\}}o^F(h(\tau)) [v_{n\in\mu xX.\phi,\epsilon}\mapsto o^d(\epsilon)]=o^{\bar F(d)}(\tau)$.
\end{proof}

\subsection{Ordinal Bounds}

\begin{lemma}
  The proof tree $d_{n\in\mu xX.\phi}$ given by Lemma \ref{thm:identity} is bounded by $\Omega$.
\end{lemma}
\begin{proof}
  We define $o^{d_{n\in\mu xX.\phi}}(\langle\rangle)=\Omega$.  Given $\sigma$ such that $d_{n\in\mu xX.\phi}(\sigma)=\mathrm{Read}^{\mathsf{ID}_1^{\infty,+}}_{[n\in\mu xX.\phi]^{n\in\mu xX.\phi,\epsilon}}$, we set $o^{d_{n\in\mu xX.\phi}}(\sigma)=v_{n\in\mu xX.\phi,\epsilon}\#1$ and, for each $\mathcal{R}$, $o^{d_{n\in\mu xX.\phi}}(\sigma\mathcal{R})=v_{n\in\mu xX.\phi,\epsilon}$.
\end{proof}

\begin{lemma}
For any $\phi$, the proof tree $d_\phi$ given by Lemma \ref{thm:ax} is bounded by $\Omega\# (rk(\phi)\cdot 2)$.
\end{lemma}
\begin{proof}
  A straightforward induction on $\phi$.
\end{proof}

\begin{lemma}\label{thm:substitution_bound}
  The proof tree $\mathrm{Subst}^{n \in \mu xX. \phi\mapsto \psi(n)}$ given by Lemma \ref{thm:substitution} is bounded by $\Omega\# rk(\psi)\cdot 2\# v_{n\in \mu xX.\phi,\langle\rangle}\# 2$ above $0$.
\end{lemma}
\begin{proof}
  We construct an ordinal bound $o^{\mathrm{Subst}^{n \in \mu xX. \phi\mapsto \psi(n)}}$ on $\dom(\mathrm{Subst}^{n \in \mu xX. \phi\mapsto \psi(n)})$ as follows.

  For any $\sigma$ such that $\mathrm{Subst}^{n \in \mu xX. \phi\mapsto \psi(n)}(\sigma)=\mathrm{Read}^{\mathsf{ID}_1^{\infty,+}}_{[\Theta_\sigma[Z\mapsto \mu xX.\phi]]^{n\in\mu xX.\phi,\epsilon,r}}$, we set $o^{\mathrm{Subst}^{n \in \mu xX. \phi\mapsto \psi(n)}}(\sigma)=rk(\psi)\cdot 2\# v_{n\in\mu xX.\phi,\epsilon}\# 2$. If $\mathcal{R}$ is any rule other than $\mathrm{Cl}_{k\in\mu xX.\phi}$ with $Z k\in\Theta_\sigma$ then we may assign $o^{\mathrm{Subst}^{n \in \mu xX. \phi\mapsto \psi(n)}}(\sigma\ulcorner\mathcal{R}\urcorner)=rk(\psi)\# v_{n\in\mu xX.\phi,\epsilon}$. (Since we have assigned $o^{\mathrm{Subst}^{n \in \mu xX. \phi\mapsto \psi(n)}}(\sigma\ulcorner\mathcal{R}\urcorner\iota)=rk(\psi)\# v_{n\in\mu xX.\phi,\epsilon\iota}\# 2$, the comparison is satisfied.)

  If $\mathcal{R}=\mathrm{Cl}_{k\in\mu xX.\phi}$ with $Z k\in\Theta_\sigma$ then we may set (recall that, above $\sigma\ulcorner\mathcal{R}\urcorner$, we have the proof tree shown in the proof of Lemma \ref{thm:substitution}):
  \begin{itemize}
  \item  $o^{\mathrm{Subst}^{n \in \mu xX. \phi\mapsto \psi(n)}}(\sigma\ulcorner\mathcal{R}\urcorner)=rk(\psi)\# v_{n\in\mu xX.\phi,\epsilon}\# 1$,
  \item $o^{\mathrm{Subst}^{n \in \mu xX. \phi\mapsto \psi(n)}}(\sigma\ulcorner\mathcal{R}\urcorner\top)=rk(\psi)\# v_{n\in\mu xX.\phi,\epsilon}$
  \item $o^{\mathrm{Subst}^{n \in \mu xX. \phi\mapsto \psi(n)}}(\sigma\ulcorner\mathcal{R}\urcorner\top \mathrm{R}\upsilon)=o^{d_{\psi(k)}}(\upsilon)$.
  \end{itemize}
  The branch $o^{\mathrm{Subst}^{n \in \mu xX. \phi\mapsto \psi(n)}}(\sigma\ulcorner\mathcal{R}\urcorner\top \mathrm{L})$ is handled as above since $\mathrm{Subst}^{n \in \mu xX. \phi\mapsto \psi(n)}(\sigma\ulcorner\mathcal{R}\urcorner\top \mathrm{L})$ is a $\mathrm{Read}$ rule.
\end{proof}

\begin{lemma}
  For any $d$ in $\mathsf{ID}_1$, the proof tree $d^\infty$ given by Theorem \ref{thm:id1_embedding} is bounded by $\Omega\cdot 2\# \omega\# n$ above $0$ for some natural number $n$.
\end{lemma}
\begin{proof}
  We define $o^{d^\infty}$ by induction on $d$. Since $d$ is finite, the induction is straightforward. In the base cases, a bound can be constructed directly or by using the previous lemma.
\end{proof}

\begin{lemma}
  \begin{enumerate}
  \item If $d$ is bounded by $\alpha$ above $\gamma$ then $\mathrm{Inverse}^\wedge_{\phi_L,\phi_R,i}(d)$ is bounded by $\alpha$ above $\gamma$.
  \item If $d$ is bounded by $\alpha$ above $\gamma$ then $\mathrm{Inverse}^\forall_{\phi(y),m}(d)$ is bounded by $\alpha$ above $\gamma$.
  \item If $d_\wedge$ is bounded by $\alpha$ above $\gamma$ and $d_\vee$ is bounded by $\beta$ above $\gamma$ then $\mathrm{Elim}^\vee_{\phi_L,\phi_R}(d_\vee,d_\wedge)$ is bounded by $\alpha\# \beta$ above $\gamma$.
  \item If $d_\forall$ is bounded by $\alpha$ above $\gamma$ and $d_\exists$ is bounded by $\beta$ above $\gamma$ then $\mathrm{Elim}^\forall_{\phi(y)}(d_\exists,d_\forall)$ is bounded by $\alpha\# \beta$ above $\gamma$.
  \item If $d_\mu$ is bounded by $\alpha$ above $\gamma$ and $d_\neg$ is bounded by $\beta$ above $\gamma$ then $\mathrm{Elim}^\mu_{n\in\mu xX.\phi}(d_\mu,d_\neg)$ is bounded by $\alpha\# \beta$ above $\gamma$.
  \item If $d$ is bounded by $\alpha$ above $\gamma$ then $\mathrm{Reduce}_n(d)$ is bounded by $\omega^\alpha$ above $\gamma$.
  \end{enumerate}
\end{lemma}
\begin{proof}
  For (1), we may take $o^{\mathrm{Inverse}^\wedge_{\phi_L,\phi_R, i}(d)}(\sigma)=o^d(\pi(\sigma))$. Similarly for (2).

  For (3), when $\pi(\sigma)$ is defined, we again set $o^{\mathrm{Elim}^\vee_{\phi_L,\phi_R}(d_\vee,d_\wedge)}(\sigma)=\alpha+o^{d_\vee}(\pi(\sigma))$. If $\pi(\sigma)$ is not defined then $\sigma$ has the form $\sigma' L\upsilon$ where $d_\vee(\pi(\sigma'))$ is $\vee\mathrm{I}^i_{{\sim}\phi_L\vee{\sim}\phi_R}$ and $\upsilon\in\dom(d_\wedge)$; in this case we set $o^{\mathrm{Elim}^\vee_{\phi_L,\phi_R}(d_\vee,d_\wedge)}(\sigma' L\upsilon)=o^{d_\wedge}(\upsilon)$. (4) and (5) are similar.

  For (6), we need to define $o^{\mathrm{Reduce}_n(d,\tau)}$ for all $\tau$ simultaneously. When $\pi_\tau(\sigma)$ is defined, we will take $o^{\mathrm{Reduce}_n(d,\tau)}(\sigma)=\omega^{o^d(\pi_\tau(\sigma))}$. Since $\alpha\ll_\gamma\beta$ implies $\omega^\alpha\ll_\gamma\omega^\beta$, this immediately ensures the order preserving property when $d(\pi_\tau(\sigma))$ is anything other than $\mathrm{Cut}_\phi$ with $rk(\phi)=n$.

  If $d(\pi_\tau(\sigma))$ is $\mathrm{Cut}_\phi$ with $rk(\phi)=n$, if $\phi$ is $\dot{R}t_1\cdots t_n$ then we take $o^{\mathrm{Reduce}_n(d,\tau)}(\sigma \top\upsilon)=o^{\mathrm{Reduce}_n(d,\pi_\tau(\sigma) b)}(\upsilon)$ for all $\upsilon$ (with $b=\top$ if $\phi$ is false and $b=\bot$ if $\phi$ is true). Since $o^{\mathrm{Reduce}_n(d,\tau) }(\sigma \top)=o^{\mathrm{Reduce}_n(d,\pi_\tau(\sigma) b)}(\langle\rangle)=\omega^{o^d(\pi_\tau(\sigma) b)}\ll_\gamma \omega^{o^d(\tau)}=o^{\mathrm{Reduce}_n(d,\tau)}(\sigma)$, the order preserving property holds. (The case where $\phi$ is $\neg\dot{R}t_1\cdots t_n$ is identical.)

  If $\phi$ is a conjunction, we take $o^{\mathrm{Reduce}_n(d,\tau)}(\sigma \top\upsilon)=o^{\mathrm{Elim}^\vee_{\phi_L,\phi_R}(\mathrm{Reduce}_n(d,\pi_\tau(\sigma) R),\mathrm{Reduce}_n(d,\pi_\tau(\sigma) L))}(\upsilon)$. We have
  \begin{align*}
    o^{\mathrm{Reduce}_n(d,\tau)}(\sigma\top)
    &=o^{\mathrm{Elim}^\vee_{\phi_L,\phi_R}(\mathrm{Reduce}_n(d,\pi_\tau(\sigma) R),\mathrm{Reduce}_n(d,\pi_\tau(\sigma) L))}(\langle\rangle)\\
    &=o^{\mathrm{Reduce}_n(d,\pi_\tau(\sigma) L)}(\langle\rangle)+o^{\mathrm{Reduce}_n(d,\pi_\tau(\sigma) R)}(\langle\rangle)\\
    &=\omega^{o^d(\pi_\tau(\sigma) L)}\#\omega^{o^d(\pi_\tau(\sigma) R)}\\
    &\ll_\gamma\omega^{o^d(\pi_\tau(\sigma))}\\
    &=o^{\mathrm{Reduce}_n(d,\tau)}(\sigma)
  \end{align*}
  as needed.
\end{proof}

\begin{lemma}\label{thm:id1_collapsing_ordinal}
  If $d$ is bounded by $\alpha$ above $\gamma$ then $D(d)$ is bounded by $\vartheta_\gamma(\alpha)$.
\end{lemma}
\begin{proof}
  When $\pi_d(\sigma)$ is defined, we set $o^{D(d)}(\sigma)=\vartheta_\gamma(o^d(\pi_d(\sigma)))$. When $d(\pi_d(\sigma))$ is not a $\mathrm{Cut}\Omega^\flat$, we have $o^d(\pi_d(\sigma\iota))\ll_\gamma o^d(\pi_d(\sigma))$, so also $o^{D(d)}(\sigma\iota)=\vartheta_\gamma(o^d(\pi_d(\sigma\iota)))< \vartheta_\gamma (o^d(\sigma))=o^{D(d)}(\sigma)$.

  When $d(\pi_d(\sigma))$ is a $\mathrm{Cut}\Omega^\flat$ rule, observe that $\bar d_\bot(D(d_\top))$ is bounded by $o^{d}(\pi_d(\sigma)\bot)[v_{n\in\mu xX.\phi,\langle\rangle}\mapsto \vartheta_\gamma(o^d(\pi_d(\sigma)\top))]$ above $\vartheta_\gamma(o^d(\pi_d(\sigma)\top))$, and therefore we may set
  \[o^{D(d)}(\sigma\top\upsilon)=\vartheta_{\vartheta_\gamma(o^d(\pi_d(\sigma)\top))}\left(o^{d}(\pi_d(\sigma)\bot\upsilon)[v_{n\in\mu xX.\phi,\langle\rangle}\mapsto \vartheta_\gamma(o^d(\pi_d(\sigma)\top))]\right).\]
  By Lemma \ref{thm:id1_D_bound}, we have $o^{D(d)}(\sigma\top)<o^{D(d)}(\sigma)$ as needed.
\end{proof}

\begin{theorem}
  If $d$ is a proof in $\mathsf{ID}_1$ whose conclusion does not contain negative occurences of $\mu xX.\phi$ then there is a proof tree $d'$ in $\mathsf{ID}_{1,<0}^{\infty,+}$ with $\Gamma(d')\subseteq\Gamma(d)$ bounded by $\vartheta \omega_k^{\Omega\cdot 2\#\omega\#n}$ for some finite $k,n$.
\end{theorem}

\section{Cut-Elimination for $\mathsf{ID}_{<\omega}$}

We can extend these methods to $\mathsf{ID}_{<\omega}$ with only a little more effort. That is, we now allow $\mu xX.\phi$ for formulas $\phi$ which themselves contain $\mu yY.\psi$ subexpressions, but we do not allow $X$ to occur free in such subexpressions.

\subsection{Infinitary Theory}

The new complication can be seen in $\mathsf{ID}_2$. In the previous section, we only applied functions to proof trees in which $\mu xX.\phi$ appeared positively---in particular, where the input proof tree did not itself contain functions.

Now, however, we might apply a function to a proof of $\mu xX. \phi(\mu yY. \psi)$ which itself contains a proof of $n\not\in\mu yY.\psi$, necessarily derived by applying an $\Omega^\flat$ rule to proof tree representing a function.

We will want to restrict ourselves to formulas in $\mathcal{L}^1_\mu$, which would lead us to a problem when defining the $\mathrm{Read}$ rule for proofs of $\mu xX.\phi(\mu yY.\psi)$. To address this, we need to modify the $\mathrm{Read}$ rule slightly.

{\tiny
\AxiomC{$\Delta(\mathcal{R})\setminus\Theta, \{[\Theta']^{n\in\mu xX.\phi,\epsilon\iota}\mid \iota\in|\mathcal{R}|, \Theta\subseteq\Theta'\}$\quad $(\mathcal{R}\in\mathfrak{T})$}
\AxiomC{$m\in\mu yY.\psi$ \quad $(m\not\in\mu yY.\psi\in\Theta, \Omega^\flat_{m\not\in\mu yY.\psi}\not\in\mathfrak{T})$}
\LeftLabel{$\mathrm{Read}^{\mathfrak{T}}_{[\Theta]^{n\in\mu xX.\phi,\epsilon}}$}
\BinaryInfC{$\Delta(\epsilon)\setminus\Theta,[\Theta]^{n\in\mu xX.\phi,\epsilon}$}
\DisplayProof}

That is, we will now sometimes have $m\not\in\mu yY.\psi\in\Theta$ while $\Omega^\flat_{m\not\in\mu yY.\psi}$ is not in the theory $\mathfrak{T}$; in this case, we need some special handling in the definition of $\bar F$, and we record a proof of $m\in\mu yY.\psi$ in an extra branch to support this.

We consider these new branches to be labeled by the conclusion $\ulcorner m\in\mu yY.\psi\urcorner$, so $|\mathrm{Read}^{\mathfrak{T}}_{[\Theta]^{n\in\mu xX.\phi,\epsilon}}|=\mathfrak{T}\cup\{m\in\mu xX.\phi\mid m\not\in\mu xX.\phi\in\Theta\text{ and }\Omega^\flat_{m\not\in\mu xX.\phi}\not\in\mathfrak{T}\}$ (with $\ulcorner \cdot\urcorner$ braces attached to all elements).

We therefore need a hierarchy of infinitary theories. All formulas in all these theories are $\mathsf{ID}_{<\omega}$ formulas.
\begin{itemize}
\item the base theory $\mathsf{ID}_0^\infty=\mathsf{ID}_0^{\infty,+}$ contains True, $\mathrm{Cut}$ over formulas of depth $0$, $\wedge\mathrm{I}$, $\vee^L\mathrm{I}$, $\vee^R\mathrm{I}$, $\omega$, $\exists\mathrm{I}$, $\mathrm{Cl}$, and Rep for $\mathsf{ID}_{<\omega}$ formulas.
\item $\mathsf{ID}_{c+1}^{\infty,+}$ extends $\mathsf{ID}_c^{\infty,+}$ by
  \begin{itemize}
  \item $\mathrm{Cut}$ for formulas of depth $c+1$,
  \item $\Omega^\flat_{n\not\in\mu xX.\phi}$ where $\phi$ has depth $c$,
  \item $\mathrm{Read}^{\mathsf{ID}_{0,<0}^{\infty,+}}_{[\Theta]^{n\in\mu xX.\phi,\epsilon}}$ rules where $\Theta\subseteq\mathcal{L}_\mu$ and $\phi$ has depth $c$,
  \end{itemize}
\item $\mathsf{ID}_{c+1}^\infty$ extends $\mathsf{ID}_c^{\infty}$ by $\mathrm{Cut}\Omega^\flat_{n\not\in\mu xX.\phi}$ where $\phi$ has depth $c$.
\end{itemize}

Note that this definition of $\mathrm{ID}^{\infty}_1$ does not quite match that in the previous section.

\subsection{Ordinal Terms}

We give this section as a somewhat informal extension of Section \ref{sec:id1_ordinal_terms}, since we will give a more formal account (of a larger system) later.

\begin{definition}
  \begin{enumerate}
  \item if $\{\alpha_0,\ldots,\alpha_{n-1}\}$ is a finite multi-set of ordinal terms in $\mathbb{H}$ and $n\neq 1$ then $\#\{\alpha_0,\ldots,\alpha_{n-1}\}$ is an ordinal term,
  \item for $c>0$, $\Omega_c$ is an ordinal term in $\mathbb{C}$,
  \item for any $n\in\mu xX.\phi,\epsilon$, $v_{n\in\mu xX.\phi,\epsilon}$ are ordinal terms in $\mathbb{V}_c$ where $c=dp(\mu xX.\phi)$,
  \item for any $\mu xX.\phi$, $v_{\mu xX.\phi}$ is an ordinal term in $\mathbb{C}$,
  \item if $\alpha$ is an ordinal term then $\omega^\alpha$ is an ordinal term in $\mathbb{H}$,
  \item if $c>0$ and $\alpha$ is an ordinal term with $FV_c(\alpha)=\emptyset$ then $\vartheta_c\alpha$ is an ordinal term in $\mathbb{SC}$.
  \end{enumerate}

  We define $FV_c(\alpha)$ to be those terms from $\bigcup_{c'\geq c}\mathbb{V}_c$ appearing in $\alpha$. We have $(\mathbb{C}\cup\bigcup_c\mathbb{V}_c)\subseteq\mathbb{SC}\subseteq\mathbb{H}$.
\end{definition}

\begin{definition}
  For $0<c$, $K_c\alpha$ is defined inductively by:
  \begin{itemize}
  \item $K_c\#\{\alpha_0,\ldots,\alpha_{n-1}\}=\bigcup_i K_c\alpha_i$,
  \item $K_c\Omega_{c'}=\left\{\begin{array}{ll}
                                 \{\Omega_{c'}\}&\text{if }c'< c\\
                                 \emptyset&\text{if }c\leq c'
                               \end{array}\right.$,
  \item $K_cv_{n\in\mu xX.\phi,\epsilon}=\left\{\begin{array}{ll}
                                 \{v_{n\in\mu xX.\phi,\epsilon}\}&\text{if }dp(\mu xX.\phi)< c\\
                                 \emptyset&\text{if }c\leq c'
                               \end{array}\right.$,
  \item $K_cv_{\mu xX.\phi}=\left\{\begin{array}{ll}
                                 \{v_{\mu xX.\phi}\}&\text{if }dp(\mu xX.\phi)\leq c\\
                                 \emptyset&\text{if }c< c'
                               \end{array}\right.$,
  \item $K_c\omega^\alpha=K_c\alpha$,
  \item $K_c\vartheta_{c'}\alpha=\left\{\begin{array}{ll}
                                          \{\vartheta_{c'}\alpha\}&\text{if }c'\leq c\\
                                          K_c\alpha&\text{if }c<c'
                                          \end{array}\right.$
  \end{itemize}

\end{definition}

The order is defined as before, noting that:
\begin{itemize}
\item $\Omega_{c'}<\Omega_c$ exactly when $c'<c$,
\item for $v\in\mathbb{V}_{c'}$, $v<\Omega_c$ if and only if $c'\leq c$,
\item for $v\in\mathbb{V}_{c'}$, $\Omega_c<v$ if and only if $c<c'$,
\item when $dp(\mu xX.\phi)=c$, $v_{\mu xX.\phi}<\Omega_c$,
\item when $dp(\mu xX.\phi)=c$ and $c'<c$, $\Omega_{c'}<v_{\mu xX.\phi}$,
\item $\vartheta_{c'}\alpha<\Omega_c$ if and only if $c'\leq c$,
\item $\Omega_c<\vartheta_{c'}\alpha$ if and only if $c<c'$.
\end{itemize}

The following definition is slightly stronger than the usual one, but useful for our purposes.
\begin{definition}
  We write $\alpha\ll^c_\gamma\beta$ if $\alpha<\beta$ and, for each $c'\leq c$ and all $\delta\in K_{c'}\alpha$, there is some $\theta\in K_{c'}\beta\cup K_{c'}\gamma$ with $\delta\leq\theta$.

  When $c=0$, we take $\ll^c_\gamma$ to simply be $<$.
\end{definition}

\begin{lemma}
  If $FV_c(\gamma)=FV_c(\delta)=\emptyset$, $\gamma,\delta<\Omega_c$, $\alpha\ll^c_\gamma\beta$, and $f:\mathbb{V}_c\rightarrow\mathbb{SC}$ is order-preserving (from $<$ to $\ll^{c-1}_0$) and has range $\lleq^{c-1}_0$-bounded by $\delta$ then $\alpha[f]\ll^c_{\gamma\#\delta}\beta[f]$.
\end{lemma}
\begin{proof}
  First we show that $\alpha[f]<\beta[f]$ by induction. The main case is when $\alpha=\vartheta_{c'}\alpha'$ and $\beta=\vartheta_{c'}\beta'$.

  If $c'\leq c$ then $\alpha',\beta'$ may not contain any elements of $\mathbb{V}_c$, so $\alpha[f]=\alpha$ and $\beta[f]=\beta$, so $\alpha<\beta$ immediately implies $\alpha[f]<\beta[f]$.

  So suppose $c<c'$. Say $\alpha'<\beta'$; by the inductive hypothesis, $\alpha'[f]<\beta'[f]$ as well. If $\zeta\in K_{c'}\alpha'[f]$ then either $\zeta=\zeta'[f]$ with $\zeta'\in K_{c'}\alpha'$ or $\zeta=\theta$ for some $\theta\in\rng(f)$. In the former case, the inductive hypothesis applies since $\zeta<\vartheta_{c'}\beta'$. In the latter case, we have $\theta\leq\delta<\Omega_c$, so $\theta<\vartheta_{c'}\beta'[f]$, as needed.

  Otherwise, $\beta'<\alpha'$, and again by the inductive hypothesis, $\beta'[f]<\alpha'[f]$. There is a $\zeta\in K_{c'}\alpha'$ so that $\vartheta_{c'}\beta'\leq \zeta$, and then by the inductive hypothesis, $\vartheta_{c'}\beta'[f]\leq\zeta[f]$.

  Now we strengthen this to $\alpha[f]\ll^c_{\max\{\gamma,\delta\}}\beta[f]$. Take $c'\leq c$ and suppose $\delta\in K_{c'}\alpha[f]$. Then $\delta\in K_{c'}\alpha\cup \bigcup_v K_{c'}f(v)$, and so is bounded by $K_{c'}\beta\cup K_{c'}\gamma\cup K_{c'}\delta$.
\end{proof}

\begin{definition}
  $\vartheta_{c,\gamma}(\alpha)=\vartheta_c(\omega^{\Omega_c\#\alpha}\#\gamma)$.
\end{definition}

\begin{lemma}
  When $FV_c(\gamma)=\emptyset$, $\gamma<\Omega_c$, $\gamma\in\mathbb{SC}$ and $\alpha\ll^c_\gamma\beta$ with $FV_c(\alpha)\cup FV_c(\beta)=\emptyset$, we have $\vartheta_{c,\gamma}(\alpha)\ll^{c-1}_0\vartheta_{c,\gamma}(\beta)$.
\end{lemma}
\begin{proof}
  We certainly have $\omega^{\Omega_c\#\alpha}\#\gamma<\omega^{\Omega_c\#\beta}\#\gamma$.

  Consider $\vartheta_c\zeta\in K_c(\omega^{\Omega_c\#\alpha}\#\gamma)=K\alpha\cup K\gamma$. Then this is bounded by some element of $K_c\beta\cup K_c\gamma$, and therefore by $\vartheta_{c,\gamma}(\beta)$.

  Now consider $\zeta\in K_{c'}\vartheta_{c,\gamma}(\alpha)$ with $c'<c$. We have $\zeta\in K_{c'}\alpha\cup K_{c'}\gamma$ as well, so again this is bounded in $K_{c'}\vartheta_{c,\gamma}(\beta)=K_{c'}\beta\cup K_{c'}\gamma$.
\end{proof}

\begin{lemma}
  When $FV_c(\gamma)=\emptyset$, $\gamma<\Omega_c$, $\gamma\in\mathbb{SC}$, $\alpha\ll^c_\gamma\beta$, $\delta\ll^c_\gamma\beta$ with $FV_c(\alpha)\cup FV_c(\beta)=\emptyset$ and $FV_c(\delta)\subseteq\{v\}$ where $v\in\mathbb{V}_c$, we have $\vartheta_{c,\vartheta_{c,\gamma}(\alpha)}(\delta[v\mapsto\vartheta_{c,\gamma}(\alpha)])\ll^{c-1}_0\vartheta_{c,\gamma}(\beta)$.
\end{lemma}
\begin{proof}
  Observe that $\delta[v\mapsto\vartheta_{c,\gamma}(\alpha)]<\beta$, since $\delta<\beta$.  Therefore we also have $\omega^{\Omega_c\#\delta[v\mapsto\vartheta_{c,\gamma}(\alpha)]}\#\gamma<\omega^{\Omega_c\#\beta}\#\gamma$.

  Any $\zeta\in K_{c} \omega^{\Omega_c\#\delta[v\mapsto\vartheta_{c,\gamma}(\alpha)]}\#\gamma$ is in $K_c\delta\cup K_c\gamma\cup\{\vartheta_{c,\gamma}(\alpha)\}$ (recall that we do not substitute inside elements of $K_c\delta$), and therefore bounded by $\vartheta_{c,\gamma}(\beta)$.

  Similarly, any element of $K_{c'}\vartheta_{c,\vartheta_{c,\gamma}(\alpha)}(\delta[v\mapsto\vartheta_{c,\gamma}(\alpha)])$ is in $K_{c'}\delta\cup K_{c'}\gamma\cup K_{c'}\vartheta_{c,\gamma}(\alpha)$, so also bounded in $K_{c'}\vartheta_{c,\gamma}(\beta)$.
\end{proof}

\begin{definition}
When $d$ is a proof tree in $\mathsf{ID}^\infty_c$, an \emph{ordinal bound above $\gamma$} on $d$ is a function $o^d:\dom(d)\rightarrow\mathrm{OT}$ so that
  \begin{itemize}
  \item if $s\sqsubseteq t$ then $o^d(t)\ll^c_\gamma o^d(s)$,
  \item for all $\sigma\in\dom(d)$, if $v_{n\in\mu xX.\phi,\epsilon}\in \bigcup_{c'}FV_{c'}(o^d(\sigma))$ then $[\Theta]^{n\in\mu xX.\phi,\epsilon}\in\Gamma(d,\sigma)$ for some $\Theta$,
  \item for all $\sigma\in\dom(d)$, if $v_{\mu xX.\phi}$ is a subterm of $o^d(\sigma)$ then some $m\not\in\mu xX.\phi$ appears as a subformula of an element of $\Gamma(d,\sigma)$.
  \end{itemize}

  We say $d$ is bounded by $\alpha$ above $\gamma$ if there is an ordinal bound above $\gamma$, $o^d$, with $o^d(\langle\rangle)\lleq_\gamma^c\alpha$.
\end{definition}

\subsection{Embedding and Cut Elimination}

Having set up our theories properly, we can now carry out the steps from the previous section with almost no change.

\begin{lemma}
  For every $\mathsf{ID}_{<\omega}$ formula $\phi$, there is a deduction $d_\phi$ in $\mathsf{ID}^{\infty,+}_{C_0}$ with $\Gamma(d_\phi)\subseteq\{\phi,{\sim}\phi\}$ bounded by $\Omega_{c}+rk(\phi)$ where $c=dp(\phi)$.
\end{lemma}
\begin{proof}
  The only new step is producing a proof of $\{n\in\mu xX.\phi,n\not\in\mu xX.\phi\}$ when $dp(\mu xX.\phi)>1$. The proof is similar to Lemma \ref{thm:identity}, except that when $d_{n\in\mu xX.\phi}(\sigma)=\mathrm{Read}^{\mathsf{ID}_0^{\infty,+}}_{[\Gamma]^{n\in\mu xX.\phi,\epsilon}}$, we could have $m\not\in\mu yY.\psi\in\Gamma$ where $dp(\mu yY.\psi)<dp(\mu xX.\phi)$. In this case we set $d_{n\in\mu xX.\phi}(\sigma\ulcorner m\in\mu yY.\psi\urcorner\tau)=d_{m\in\mu yY.\psi}(\tau)$ for all $\tau$.
\end{proof}

\begin{lemma}\label{thm:substitution2}
  For any formula $\psi(y)$ of depth $\leq C_0$ with only the free variable $y$ and $\mu xX.\phi$ of depth $c$, there is a proof tree $\mathrm{Subst}^{n\in\mu xX.\phi\mapsto\psi(n)}$ in $\mathsf{ID}^{\infty}_{\max\{c,C_0\}}$ with conclusion
  \[[n\in\mu xX.\phi]^{n\in\mu xX.\phi,\langle\rangle,\mu xX.\phi},\psi(n),\exists y\,\phi(y,\psi)\wedge{\sim}\psi(y)\]
and bounded by $\Omega_{C_0}\# rk(\psi)\cdot 2\# v_{n\in\mu xX.\phi,\langle\rangle}\# 2$.
\end{lemma}
\begin{proof}
  This is unchanged from Lemma \ref{thm:substitution} since the theory the $\mathrm{Read}$ rules branch over is essentially unchanged. (Note that we never have any $n\not\in\mu xX.\phi\in\Theta'_\sigma$.)
\end{proof}

\begin{theorem}\label{thm:idomega_embedding}
  Let $d$ be a deduction in $\mathsf{ID}_{<\omega}$ so that $\Gamma(d)$ has free variables contained in $x_1,\ldots,x_n$. There are some $C_0,k$ so that, for any numerals $m_1,\ldots,m_n$, there is a proof tree $d^\infty$ in $\mathsf{ID}^\infty_{C_0,<k}$ with $\Gamma(d^\infty)\subseteq\Gamma(d)[x_i\mapsto m_i]$ and bounded by $\Omega_{C_0}\cdot 3\#\omega\# k$.
\end{theorem}
\begin{proof}
  As in the proof of Theorem \ref{thm:id1_embedding}. 
\end{proof}

\begin{theorem}
  \begin{enumerate}
  \item   Let $\eta$ be an atomic formula of the form $\dot{R}m_1\cdots m_n$ and let $\phi\in\{\sigma,\neg\sigma\}$ be false. If $d$ is a $\mathsf{ID}^\infty_{<\omega,0}$ proof tree then $\phi\not\in\Gamma(d)$.
  \item   For any formulas $\phi_L,\phi_R$ with rank $\leq n$ and $i\in\{L,R\}$ and any $\mathsf{ID}^\infty_{c,<n}$ proof tree $d$ with bound $\alpha$, there is a $\mathsf{ID}^\infty_{c,<n}$ proof tree $\mathrm{Inverse}^\wedge_{\phi_L,\phi_R,i}(d)$ with $\Gamma(\mathrm{Inverse}^\wedge_{\phi_L,\phi_R,i}(d))\subseteq(\Gamma(d)\setminus\{\phi_L\wedge\phi_R\})\cup\{\phi_i\}$ and bounded by $\alpha$.
  \item   For any formula $\phi(y)$ with rank $\leq n$ whose only free variable is $\phi$, any $m\in\mathbb{N}$, and any $\mathsf{ID}^\infty_{c,<n}$ proof tree $d$ with bound $\alpha$, there is a $\mathsf{ID}^\infty_{c,<n}$ proof tree $\mathrm{Inverse}^\forall_{\phi(y),m}(d)$ with $\Gamma(\mathrm{Inverse}^\forall_{\phi(y),m}(d))\subseteq(\Gamma(d)\setminus \{\forall x\,\phi(x)\})\cup\{\phi(m)\}$ and bounded by $\alpha$
  \item   For any formulas $\phi_L,\phi_R$ with rank $\leq n$ and any $\mathsf{ID}^\infty_{c,<n}$ proof trees $d_\wedge$ and $d_\vee$ bounded by $\alpha,\beta$ respectively, there is a $\mathsf{ID}^\infty_{c,<n}$ proof tree $\mathrm{Elim}^\vee_{\phi_L,\phi_R}(d_\vee,d_\wedge)$ with $\Gamma(\mathrm{Elim}^\vee_{\phi_L,\phi_R}(d_\vee,d_\wedge))\subseteq (\Gamma(d_\wedge)\setminus\{\phi_L\wedge\phi_R\})\cup(\Gamma(d_\vee)\setminus\{{\sim}\phi_L\vee{\sim}\phi_R\})$ and bounded by $\alpha\# \beta$.
  \item   For any formula $\phi(y)$ with rank $\leq n$ and any $\mathsf{ID}^\infty_{c,<n}$ proof trees $d_\forall$ and $d_\exists$ bounded by $\alpha,\beta$ respectively, there is a $\mathsf{ID}^\infty_{c,<n}$ proof tree $\mathrm{Elim}^\exists_{\phi(y)}(d_\exists,d_\forall)$ with $\Gamma(\mathrm{Elim}^\exists_{\phi(y)}(d_\exists,d_\forall))\subseteq (\Gamma(d_\wedge)\setminus\{\forall x\,\phi(x)\})\cup(\Gamma(d_\vee)\setminus\{\exists x\,{\sim}\phi(x))$ and bounded by $\alpha\#\beta$.
  \item For any atomic formula of the form $n\in\mu xX.\phi$ with depth $\leq c$ and $\mathsf{ID}_{c,<0}^{\infty}$ proof trees $d_\mu$ and $d_\neg$ bounded by $\alpha,\beta$ respectively, there is a $\mathsf{ID}_{c,<0}^\infty$ proof tree $\mathrm{Elim}^{\mu}_{n\in\mu xX.\phi}(d_\mu,d_\neg)$ with $\Gamma(\mathrm{Elim}^{\mu}_{n\in\mu xX.\phi}(d_\mu,d_\neg))\subseteq(\Gamma(d_\mu)\setminus\{n\in \mu xX.\phi\})\cup(\Gamma(d_\neg)\setminus\{n\not\in\mu xX.\phi\})$ and bounded by $\alpha\#\beta$. \label{case_elim}
  \item If $d$ is a proof tree in $\mathsf{ID}_{c,<n+1}^{\infty}$ bounded by $\alpha$ then there is a proof tree $\mathrm{Reduce}_n(d)$ in $\mathsf{ID}_{c,<n}^\infty$ with $\Gamma(\mathrm{Reduce}_n(d))\subseteq\Gamma(d)$ and bounded by $\omega^\alpha$.
  \end{enumerate}
\end{theorem}
\begin{proof}
  The proofs are largely unchanged except for (\ref{case_elim}).

  In the definition of $\mathrm{Elim}^\mu_{n\in\mu xX.\phi}(d_\mu,d_\neg)$, we need to address the possibility that $n\not\in\mu xX.\phi\in\Delta(d_\neg(\pi(\sigma)))$ but $d_{\neg}(\pi(\sigma))$ is not $\Omega^\flat_{n\not\in\mu xX.\phi}$, because $d_\neg(\pi(\sigma))$ could now be $\mathrm{Read}^{\mathsf{ID}_{0,<0}^{\infty}}_{[\Theta]^{m\in\mu yY.\psi,\epsilon'}}$ with $n\not\in\mu xX.\phi\in\Delta(\epsilon')\setminus\Theta$.

  In this case, we set $\mathrm{Elim}^\mu_{n\in\mu xX.\phi}(d_\mu,d_\neg)(\sigma)=\mathrm{Read}^{\mathsf{ID}_{0,<0}^{\infty}}_{[\Theta,n\not\in\mu xX.\phi]^{m\in\mu yY.\psi,\epsilon'}}$. For each $\iota\in|\mathrm{Read}^{\mathsf{ID}_{0,<0}^{\infty}}_{[\Theta]^{m\in\mu yY.\psi,\epsilon'}}|$ we set $\pi(\sigma\iota)=\pi(\sigma)\iota$ and continue as before.

  However there is a new element $\ulcorner n\in\mu xX.\phi\urcorner\in|\mathrm{Read}^{\mathsf{ID}_{0,<0}^{\infty}}_{[\Theta]^{m\in\mu yY.\psi,\epsilon'}}|\setminus|\mathrm{Read}^{\mathsf{ID}_{0,<0}^{\infty}}_{[\Theta]^{m\in\mu yY.\psi,\epsilon'}}|$; we set $\mathrm{Elim}^\mu_{n\in\mu xX.\phi}(d_\mu,d_\neg)(\sigma\ulcorner n\in\mu xX.\phi\urcorner\tau)=d_\mu(\tau)$ for all $\tau$.
  
\end{proof}

\subsection{Lifting}

Consider what collapsing, the analog of Theorem \ref{thm:id1_collapsing}, will look like. We will begin with a proof tree in $\mathsf{ID}_{c+1,<0}^{\infty}$ and need to collapse $\mathrm{Cut}\Omega^\flat$ rules. After collapsing, we get a proof in $\mathsf{ID}_{c,<0}^{\infty}$---we expect to still have $\mathrm{Read}$ and $\Omega^\flat$ inferences, both for lower depth formulas (because there might be lower depth $\mathrm{Cut}\Omega^\flat$ rules) and because we will now allow negative occurences even of depth $c+1$ formulas.

When we wish to collapse a $\mathrm{Cut}\Omega^\flat$ rule, this presents a problem---on the left, collapsing only gives us a proof tree in $\mathsf{ID}_{c,<0}^{\infty}$, while on the right we have a function which expects a proof tree in $\mathsf{ID}_{0,<0}^{\infty,+}$ as input.

The solution is that we will just apply our function anyway, having it act as the identity whenever it encounters ``unexpected'' rules. We will need this technique again, so we isolate it in general---when we have a locally defined function defined on a theory $\mathfrak{T}$, we can \emph{lift} the function to some $\mathfrak{T}'\supseteq\mathfrak{T}$, and the new function will have the same conclusion under reasonable assumptions. In the definition below, we further assume that $\mathfrak{T}'$ includes new $\Omega^\flat$ rules, and that we want to modify our behavior on those.

We can write this as a transformation on proof trees, though the result is long and rather technical; it is more convenient (and clearer) to describe the translation directly on the function:

\[\bar F^{\uparrow\mathfrak{T}^*}(\{d_{\epsilon}\})=
  \left\{\begin{array}{l}
           \mathrm{Rep}(\bar F_{d_{\epsilon_0}(\langle\rangle)}(\{d_\epsilon\}\cup\{d_{\epsilon_0\iota}\}_\iota))\\
           \quad \quad \quad \quad\text{if }F(\langle\rangle)=\mathrm{Read}^{\mathfrak{T}}_{[\Theta]^{n\in\mu xX.\phi,\epsilon_0}}\text{ and }d_{\epsilon_0}(\langle\rangle)\in\mathfrak{T}\\
           \mathcal{R}(\{\bar F(\{d_\epsilon\}_{\epsilon\neq\epsilon_0}\cup\{\epsilon_0\mapsto d_{\epsilon_0\iota}\})\}_{\iota\in|\mathcal{R}|}) \\
           \quad \quad \quad \quad\text{if }F(\langle\rangle)=\mathrm{Read}^{\mathfrak{T}}_{[\Theta]^{n\in\mu xX.\phi,\epsilon_0}}\text{, and}\\
           \quad \quad \quad \quad d_{\epsilon_0}(\langle\rangle)=\mathcal{R}\not\in\mathfrak{T}, \\
           \quad \quad \quad \quad \text{and }\mathcal{R}\text{ is not a Read or }\Omega^\flat_{n\in\mu xX.\phi}\text{ rule with }n\not\in\mu xX.\phi\in\Theta\\
           \mathrm{Read}^{\mathfrak{T}''}_{[\Theta''\cup\Theta]^{m\in\mu yY.\psi,\rho}}(\{\bar F(\{d_\epsilon\}_{\epsilon\neq\epsilon_0}\cup\{\epsilon_0\mapsto d_{\epsilon_0\iota}\})\}_{\iota\in|\mathrm{Read}^{\mathfrak{T}''}_{[\Theta''\cup\Theta]^{m\in\mu yY.\psi,\rho}}|}) \\
           \quad \quad \quad \quad\text{if }F(\langle\rangle)=\mathrm{Read}^{\mathfrak{T}}_{[\Theta]^{n\in\mu xX.\phi,\epsilon_0}}\text{, and}\\
           \quad \quad \quad \quad d_{\epsilon_0}(\langle\rangle)=\mathrm{Read}^{\mathfrak{T''}}_{[\Theta'']^{m\in\mu yY.\psi,\rho}}\not\in\mathfrak{T}\\
\mathrm{Cut}\Omega^\flat(\bar F_{n\in\mu xX.\phi}(\{d_\epsilon\}_{\epsilon}),\bar F(\{d_\epsilon\}_{\epsilon\neq\epsilon_0}\cup\{\epsilon_0\mapsto d_{\epsilon_0\bot}\})) \\
           \quad \quad \quad \quad\text{if }F(\langle\rangle)=\mathrm{Read}^{\mathfrak{T}}_{[\Theta]^{n\in\mu xX.\phi,\epsilon_0}}\text{, and}\\
           \quad \quad \quad \quad d_{\epsilon_0}(\langle\rangle)=\Omega^\flat_{n\in\mu xX.\phi}\not\in\mathfrak{T}\text{ with }n\not\in\mu xX.\phi\in\Theta\\           
           \mathcal{R}(\{\bar F_\iota(\{d_\epsilon\})\}_\iota) \\
           \quad \quad \quad \quad\text{if }F(\langle\rangle)=\mathcal{R}\neq \mathrm{Read}^{\mathfrak{T}}_{[\Theta]^{n\in\mu xX.\phi,\epsilon_0}}
         \end{array}\right..\]
     
In the new cases---the second through fourth---note that $\bar F$ (and $\bar F_{n\in\mu xX.\phi}$) have the same set of inputs $\Upsilon_{n\in\mu xX.\phi}(\bar F)$, but we are now using the branch $d_{\epsilon_0\iota}$ for the input labeled by $\epsilon_0$.

Note that the range of $F^{\uparrow\mathfrak{T}^*}$ can grow (beyond just adjusting the Read rules used): it now includes all $\mathrm{Cut}\Omega^\flat_{n\in\mu xX.\phi}$ axioms such that $\Omega^\flat_{n\in\mu xX.\phi}\in\mathfrak{T}^*$ and $n\not\in\mu xX.\phi$ is removed by some $\mathrm{Read}$ rule in the proof.

\begin{definition}
  We say $\mathfrak{T}^*$ is \emph{rule-for-rule conservative} over $\mathfrak{T}$ in $F$ if, for all $\mathcal{R}\in\mathfrak{T}^*\setminus\mathfrak{T}$ other than $\mathrm{Read}$ rules and all $\sigma\in\dom(F)$ so that $F(\sigma)=\mathrm{Read}^{\mathfrak{T}}_{[\Theta]^{\phi_F}}$, we have $\Delta(\mathcal{R})\cap\Theta=\emptyset$.
\end{definition}

\begin{lemma}
  If $F$ is a locally defined function from $\mathfrak{T}$ to $\mathfrak{T}'$ and $\mathfrak{T}^*\supseteq\mathfrak{T}$ is rule-for-rule conservative over $\mathfrak{T}$ then $\Gamma(F^{\uparrow\mathfrak{T}^*})\cap\mathcal{L}_\mu\subseteq\Gamma(F)$.
\end{lemma}
The behavior of bracketed formulas is slightly more complicated, but will not matter for the cases we need.
\begin{proof}
  Suppose $\phi\in\Gamma(F^{\uparrow\mathfrak{T}^*})$, so there is some $\sigma$ with $\phi\in\Delta(F^{\uparrow\mathfrak{T}^*}(\sigma))$ and, for all $\tau\iota\sqsubseteq\sigma$, $\phi\not\in\Delta_\iota(F^{\uparrow\mathfrak{T}^*}(\tau))$. If $F^{\uparrow\mathfrak{T}^*}(\sigma)=F(\pi(\sigma))$ then we certainly have $\phi\in\Gamma(F)$, so the case we need to consider is where $F^{\uparrow\mathfrak{T}^*}(\sigma)\neq F(\pi(\sigma))$, and therefore $\sigma=\sigma' \ulcorner\mathcal{R}\urcorner$ where $F^{\uparrow\mathfrak{T}^*}(\sigma')$ is $\mathrm{Read}^{\mathfrak{T}^*}_{[\Theta]^{\phi_F,\rho_\sigma(\epsilon)}}$ rule $\mathcal{R}\in\mathfrak{T}^+\setminus\mathfrak{T}$.

  If $\mathcal{R}$ is not a $\mathrm{Read}$ rule then we have $F^{\uparrow\mathfrak{T}^*}(\sigma' \ulcorner\mathcal{R}\urcorner)=\mathcal{R}$ and $\Delta(\mathcal{R})\cap\Theta=\emptyset$ because $\mathfrak{T}^*$ is rule-for-rule conservative. If $\mathcal{R}$ is a $\mathrm{Read}$ rule then we have $\Delta(F^{\uparrow\mathfrak{T}^*}(\sigma' \ulcorner\mathcal{R}\urcorner))\cap\Theta=\emptyset$ by the construction. In either case, $\phi\in\Delta_{\ulcorner\mathcal{R}\urcorner}(F^{\uparrow\mathfrak{T}^*}(\sigma'))$, contradicting the choice of $\sigma$.
\end{proof}

\begin{lemma}
If $F$ is bounded by $\alpha$ then $F^{\uparrow}$ is bounded by $\alpha$ as well.
\end{lemma}
\begin{proof}
Fix some ordinal bound $o^F$ on $F$. We define an ordinal bound $o^{F^{\uparrow}}$ on $F^{\uparrow}$ by setting $o^{F^\uparrow}(\sigma)=o^F(\pi(\sigma))[v_{\epsilon}\mapsto v_{\rho_\sigma(\epsilon)}]$.
\end{proof}

\subsection{Collapsing}

Allowing $\mu$-expressions to appear negatively even after collapsing creates a new complications. For instance, even with $\mathsf{ID}^\infty_{m}$ proof trees, the collapsed tree will not be well-founded if the conclusion contains some $n\not\in\mu xX.\phi$ formula. However we still wish to represent some improvement in our ordinal bounds---in particular, our commitment that the proof tree not contain $\mathrm{Cut}\Omega^\flat$ rules.

We use the new variables $v_{\mu xX.\phi}$ for this purpose: after collapsing, we require that $\Omega_c$ no longer appear free, but do permit $v_{\mu xX.\phi}$. It follows that, as we collapse, we have to replace $\Omega_c$ with $\Omega_c\#{\scalebox{1.5}{\#}}\{v_{\mu xX.\phi}\}$ where the $\#$ ranges over those $\mu xX.\phi$ of depth $c$ appearing negatively in $\Gamma(d)$.

\begin{theorem}\label{thm:idomega_collapsing}
  Suppose $d$ is a proof tree in $\mathsf{ID}_{c+1,<0}^\infty$ bounded by $\alpha$ above $\gamma$. Let $\beta$ be the ordinal term $\Omega_{c+1}\#{\scalebox{1.5}{\#}}\{v_{\mu xX.\phi}\}$ where the $\scalebox{1.5}{\#}$ ranges over those $\mu xX.\phi$ of depth $c+1$ appearing negatively in $\Gamma(d)$.  Then there is a proof tree $D_{c}(d)$ in $\mathsf{ID}^\infty_{c,<0}$ with $\Gamma(D_{c+1}(d))\subseteq\Gamma(d)$ and bounded by $\vartheta_{c+1,\gamma}(\alpha[\Omega_{c+1}\mapsto\beta])$ above $0$.
\end{theorem}
\begin{proof}
  As in the proof of Theorem \ref{thm:id1_collapsing}, we define $D_c(d)(\sigma)$ by induction on $|\sigma|$ for all $d$ simultaneously, as well as partial functions $\pi_d:\dom(D_c(d))\rightarrow\dom(d)$. The assignment of ordinals, $o^{D_c(d)}$, is similar to Lemma \ref{thm:id1_collapsing_ordinal}, setting $o^{D_c(d)}(\sigma)=\vartheta_{c+1,\gamma} o^d(\sigma)[\Omega_{c+1}\mapsto\beta]$ when $\pi_d(\sigma)$ is defined, and copying over $o^{D_c(d)}(\sigma\top\upsilon)=o^{D_c(\bar d^{\uparrow \mathsf{ID}^\infty_{c,<0}}_\bot(D_m(d_\top)))}(\upsilon)$ after removing a $\mathrm{Cut}\Omega^\flat$ at $d(\pi_d(\sigma))$.

  When $d(\pi_d(\sigma))$ is not $\mathrm{Cut}\Omega^\flat_{n\in\mu xX.\phi}$ where $\phi$ has depth $c+1$, we need to check that $o^{D_m(d)}(\sigma\iota)=\vartheta_{c+1,\gamma}(o^d(\pi_d(\sigma\iota))[\Omega_{c+1}\mapsto\beta])\ll^{c}_0\vartheta_{c+1,\gamma} (o^d(\pi_d(\sigma))[\Omega_{m+1}\mapsto\beta])=o^{D_c(d)}(\sigma)$.

  The significant case is when $d(\pi_d(\sigma))=\mathrm{Cut}\Omega^\flat_{n\in\mu xX.\phi}$ where $\phi$ has depth $c+1$. We set $D_m(d)(\sigma)=\mathrm{Rep}$ and $D_c(d)(\sigma\top\upsilon)=D_c(\bar d^{\uparrow \mathsf{ID}^\infty_{c,<0}}_\bot(D_c(d_\top)))(\upsilon)$---the only difference is that $\bar d_\bot$ is only a function on proofs in $\mathsf{ID}^{\infty,+}_{c+1,<0}$, and we use lifting to extend it to a function on proofs in $\mathsf{ID}^{\infty,+}_{c+1,<0}\cup\mathsf{ID}^\infty_{c,<0}$ as needed.

  The ordinal bound $o^{D_c(d)}(\sigma\top)\ll^{c}_0 o^{D_m(d)}(\sigma)$ is as above. We can then check that
  \[o^{D_c(d)}(\sigma\top)=\vartheta_{c+1,\gamma}\left( o^d(\pi_d(\sigma)\bot)[v_{n\in\mu xX.\phi,\langle\rangle}\mapsto \vartheta_{m+1,\gamma} o^{D_m(d)}(\sigma\top)[\Omega_{c+1}\mapsto\beta],\Omega_{c+1}\mapsto\beta]\right)\]
  is $\ll^{c}_0 o^{D_c(d)}(\sigma)$.
\end{proof}

\begin{theorem}
    If $d$ is a proof in $\mathsf{ID}_c$ whose conclusion does not contain negative occurences of any $\mu xX.\phi$ then there is a proof tree $d'$ in $\mathsf{ID}_{0,<0}^{\infty,+}$ with $\Gamma(d')\subseteq\Gamma(d)$ bounded by $\vartheta_{0,0} \omega_k^{\Omega_c\cdot 2\#\omega\#n}$ for some finite $k,n$.
\end{theorem}

\section{$\mathsf{ID}_{<\bOmega+\omega}$}

\subsection{Theories}

As a final stepping stone, we consider a theory that really requires our new machinery: the theory $\mathsf{ID}_{<\bOmega+\omega}$ which allows formulas like 
\[n\in \mu xX. \phi(x,X,\mu yY. \phi'(x,y,X,Y))\]
where $\phi$ has depth $<\omega$. Here, and throughout this section, we adopt the convention that when $\mu xX.\phi$ is a $\mu$-expression with depth $\bOmega$, we write $\phi'$ for the nested $\mu$-expression. (We adopt the---harmless---restriction that there is always a single such $\phi'$, since any $\mu$-expression of depth $\bOmega+n$ is equivalent to one of this form.)

More precisely, we fix a natural number $C_0$ and consider the theory $\mathsf{ID}_{<\bOmega+C_0}$ in which we allow formulas of depth $[0,C_0)\cup[\bOmega,\bOmega+C_0)$. The theories $\mathsf{ID}_{\bOmega+C_0}$, of course, exhaust $\mathsf{ID}_{<\bOmega+\omega}$. (A single nested fixed point $\mu xX. \phi(x,X,\mu yY. \phi'(x,y,X,Y))$, where $\phi,\phi'$ each have depth $0$, already suffices to capture the full strength of the theory---the substantive part of the restriction to $C_0$ is that it prohibits nesting of induction rules which would violate the bound.)

\begin{definition}\ 
  \begin{itemize}
  \item $\mathsf{ID}_0^{\infty,+}$ is the arithmetic part: the theory containing True, $\mathrm{Cut}$ for formulas of depth $0$, $\wedge\mathrm{I}$, $\vee^L\mathrm{I}$, $\vee^R\mathrm{I}$, $\omega$, $\exists\mathrm{I}$, $\mathrm{Cl}$, and Rep for closed proper formulas of depth either $<C_0$ in $[\bOmega,C_0)$,
\item $\mathsf{ID}_{c+1}^{\infty,+}$ extends $\mathsf{ID}_c^{\infty,+}$ by
  \begin{itemize}
  \item $\mathrm{Cut}$ for formulas of depth $c+1$,
  \item $\Omega^\flat_{n\not\in\mu xX.\phi}$ where $\mu xX.\phi$ has depth $c+1$,
  \item $\mathrm{Read}^{\mathsf{ID}_{0,<0}^\infty}_{[\Theta]^{n\in\mu xX.\phi,\epsilon}}$ rules where $\phi$ has depth $c$ and $\Theta\subseteq\mathcal{L}_\mu$,
  \end{itemize}
\item $\mathsf{ID}_{\bOmega}^{\infty,+}=\mathsf{ID}_{C_0-1}^{\infty,+}$,
\item $\mathsf{ID}^{\infty}_{\bOmega}$ extends $\mathsf{ID}_{C_0-1}^{\infty,+}\cup\mathsf{ID}_{\bOmega+C_0-1}^{\infty,+}$ by
  \begin{itemize}
  \item $\Omega^\flat_{n\not\in\mu xX.\phi}$ and $\mathrm{Cut}\Omega^\flat_{n\not\in\mu xX.\phi}$ where $\mu xX.\phi$ has depth $\bOmega$,
  \item $\mathrm{Read}^{\mathsf{ID}_{\bOmega+C_0-1,<0}^{\infty,+}}_{[\Theta]^{n\in\mu xX.\phi,\epsilon}}$ rules where $\mu xX.\phi$ has depth $\bOmega$ and $\Theta\subseteq\mathcal{L}^1_\mu$.%
  \end{itemize}
\item for $c$ a successor, $\mathsf{ID}_c^\infty$ extends $\mathsf{ID}_{c-1}^\infty\cup \mathsf{ID}_c^{\infty,+}$ by $\mathrm{Cut}\Omega^\flat_{n\not\in\mu xX.\phi}$ where $\phi$ has depth $c-1$.
\end{itemize}
\end{definition}

\subsection{Substitution}

\begin{lemma}
  For every proper closed formula $\phi$, there is a proof tree $d_\phi$ in $\mathsf{ID}^\infty_0$ with $\Gamma(d_\phi)\subseteq\{\phi,{\sim}\phi\}$.
\end{lemma}

For successor $c$, we have the same substitution function from Lemma \ref{thm:substitution2}. For $c=\bOmega$, we need some additional work.

\begin{lemma}\label{thm:substitution3}
  For any formula $\psi(y)$ with only the free variable $y$ and $\mu xX.\phi$ of depth $\bOmega$ such that $\phi(\psi)$ has depth $<\bOmega+C_0$, there is a proof tree $\mathrm{Subst}^{n \in \mu xX. \phi\mapsto \psi(n)}$ in $\mathsf{ID}^\infty_{\bOmega+C_0-1}$ with conclusion
  \[[n\in \mu xX.\phi]^{n\in\mu xX.\phi,\langle\rangle},\psi(n),\exists y\,\phi(y,\psi)\wedge{\sim}\psi(y).\]
  
\end{lemma}
The first new situation that is forced on us is that we might have to change an $\Omega^\flat$ rule:

\AxiomC{$d'$}
\noLine
\UnaryInfC{$\vdots$}
\noLine
\UnaryInfC{$[n\in\mu yY.\rho(\mu xX.\phi)]^{n\in\mu yY.\rho(\mu xX.\phi),\langle\rangle}$}
\LeftLabel{$\Omega^\flat$}
\UnaryInfC{$n\not\in\mu yY.\rho(\mu xX.\phi)$}
\DisplayProof
$\mapsto$
\AxiomC{$\mathrm{IH}(d')\}$}
\noLine
\UnaryInfC{$\vdots$}
\noLine
\UnaryInfC{$[n\in\mu yY.\rho(\psi)]^{n\in\mu yY.\rho(\psi),\langle\rangle}$}
\LeftLabel{$\Omega^\flat$}
\UnaryInfC{$n\not\in\mu yY.\rho(\psi)$}
\DisplayProof

On the one hand, this is a simple change, replacing a rule with a different case of the same rule. On the other hand, it forces a substantial further complication on us: we might now have bracketed formulas which need to be replaced. This means we also need to apply the substitution operation to $\mathrm{Read}$ rules, replacing them with alternate $\mathrm{Read}$ rules. We would like to take a step that looks like this:

\AxiomC{$d'_{\mathcal{R}}$}
\noLine
\UnaryInfC{$\vdots$}
\noLine
\UnaryInfC{$\Delta(\mathcal{R})\setminus\Upsilon (\mu xX.\phi),\{[\Upsilon'(\mu xX.\phi)]^{n\in\mu yY.\rho(\mu xX.\phi),\nu\iota}\}$}
\LeftLabel{$\mathrm{Read}^{\mathsf{ID}^{\infty,+}_{<\bOmega+c'+1}}_{[\upsilon(\mu xX.\phi)]^{n\in\mu yY.\rho(\mu xX.\phi),\nu}}$}
\UnaryInfC{$\Delta(\nu)\setminus\Upsilon(\mu xX.\phi), [\upsilon(\mu xX.\phi)]^{n\in\mu yY.\rho(\mu xX.\phi),\nu}$}
\DisplayProof

and convert it to one that looks like this:

\AxiomC{$\mathrm{IH}(d'_{???})$}
\noLine
\UnaryInfC{$\vdots$}
\noLine
\UnaryInfC{$\Delta(\mathcal{R})\setminus\upsilon(\psi),\{[\upsilon'(\psi)]^{n\in\mu yY.\rho(\psi),\nu'\iota}\}$}
\LeftLabel{$\mathrm{Read}^{\mathsf{ID}^{\infty,+}_{<c+c
      +1}}_{[\Upsilon(\psi)]^{n\in\mu yY.\rho(\psi),\nu'}}$}
\UnaryInfC{$\Delta(\nu')\setminus\upsilon(\psi), [\upsilon(\psi)]^{n\in\mu yY.\rho(\mu xX.\phi),\nu'}$}
\DisplayProof

This leads to a yet further complication (the last one, fortunately): the branches of our new $\mathrm{Read}$ rule do not directly align to the branches of our old $\mathrm{Read}$ rule---we do not necessarily want to use $\mathrm{IH}(d'_{\mathcal{R}})$ for the branch indexed by $\mathcal{R}$: the rule $\mathcal{R}$ might introduce a formula in $\Delta'(\psi)$, in which case we want to instead use $\mathrm{IH}(d'_{\mathcal{R}'})$, where $\mathcal{R}'$ introduces the corresponding formula in $\Delta'(\mu xX.\phi)$.

\begin{proof}
  The proof is very similar to Lemma \ref{thm:substitution2} (and therefore \ref{thm:substitution}), with some extra accounting to deal with the new cases.

  The relationship between the formulas we are removing and those we are replacing them with is no longer quite as simple, so we replace $\Theta_\sigma[Z\mapsto\mu xX.\phi]$ with $\Theta_\sigma^{\mathrm{old}}$ and $\Theta_\sigma[Z\mapsto\psi]$ with $\Theta_\sigma^{\mathrm{new}}$. We still have $\Theta'_\sigma=\Theta_\sigma^{\mathrm{old}}\setminus\Theta_\sigma^{\mathrm{new}}$.

  Once again, we need to consider the case where $\mathrm{Subst}^{n \in \mu xX. \phi\mapsto \psi(n)}(\sigma)=\mathrm{Read}^{\mathsf{ID}^{\infty,+}_{C_0-1,<0}}_{[\Theta'_\sigma]^{n\in\mu xX.\phi,\epsilon}}$ and we need to consider the possible choices for $\mathcal{R}\in \mathsf{ID}^{\infty,+}_{C_0-1,<0}$. If $\Delta(\mathcal{R})\cap\Theta'_\sigma=\emptyset$, the argument remains unchanged.

  Suppose $\mathcal{R}$ is an $\Omega^\flat_{m\not\in\mu yY.\phi'[Z\mapsto\mu xX.\phi]}$ with $\Delta(\mathcal{R})\cap\Theta'_\sigma\neq\emptyset$. In this case we may take $\mathrm{Subst}^{n \in \mu xX. \phi\mapsto \psi(n)}(\sigma\ulcorner\mathcal{R}\urcorner)=\Omega^\flat_{m\not\in\mu yY.\phi'[Z\mapsto\psi]}$. We then set $\Theta^{\mathrm{old}}_{\sigma\ulcorner\mathcal{R}\urcorner\bot}=\Theta^{\mathrm{old}}_\sigma\cup\{[m\in\mu yY.\phi']^{m\in\mu yY.\phi'(\mu xX.\phi),\langle\rangle}\}$ and $\Theta^{\mathrm{new}}_{\sigma\ulcorner\mathcal{R}\urcorner\bot}=\Theta^{\mathrm{new}}_\sigma\cup\{[m\in\mu yY.\phi'(\psi)]^{m\in\mu yY.\phi'(\psi),\langle\rangle}\}$, and then we may set $\mathrm{Subst}^{n \in \mu xX. \phi\mapsto \psi(n)}(\sigma\ulcorner\mathcal{R}\urcorner\bot)=\mathrm{Read}^{\mathsf{ID}^{\infty,+}_{<\bOmega}}_{[\Theta'_{\sigma\ulcorner\mathcal{R}\urcorner\iota}]^{n\in\mu xX.\phi,\epsilon\bot}}$.

  The presence of $[m\in\mu yY.\phi'(\mu xX.\phi)]^{m\in\mu yY.\phi'(\mu xX.\phi),\langle\rangle}$ in $\Theta_\sigma$ means that we further have the possibility that $\Delta(\mathcal{R})\cap\Theta'_\sigma\neq\emptyset$ because $\mathcal{R}$ is some $\mathrm{Read}^{\mathsf{ID}^{\infty}_{0,<0}}_{[\Upsilon]^{m\in\mu yY.\phi'(\mu xX.\phi),\nu}}$ rule with $[\Upsilon]^{m\in\mu yY.\phi'(\mu xX.\phi),\nu}\in\Theta'_\sigma$. There is a corresponding $[\Upsilon']^{m\in\mu yY.\phi'(\psi),\nu'}$ in $\Theta_\sigma^{\mathrm{new}}$. We will maintain inductively that there is an $\Upsilon^*$ so that $\Upsilon=\Upsilon^*[Z\mapsto\mu xX.\phi]$ while $\Upsilon'=\Upsilon^*[Z\mapsto\psi]\cup\Upsilon^+$ with $\Upsilon^+\subseteq\Theta_\sigma^{\mathrm{new}}$.

  We set $\mathrm{Subst}^{n \in \mu xX. \phi\mapsto \psi(n)}(\sigma\ulcorner\mathcal{R}\urcorner)=\mathrm{Read}^{\mathsf{ID}^{\infty}_{0,<0}}_{[\Upsilon']^{m\in\mu yY.\phi'[Z\mapsto\psi],\theta(\nu)}}$

  We face an additional complication: we can no longer simply identify the premises of $\mathrm{Subst}^{n \in \mu xX. \phi\mapsto \psi(n)}(\sigma\ulcorner\mathcal{R}\urcorner)$ with the premises of $\mathcal{R}$. So we need a more complicated argument that matches our new premises to suitable old premises. %

  We first consider the rule premises. Consider some $\mathcal{R}'$ in $\mathsf{ID}_{0,<0}^{\infty,+}$.

  First, suppose $\Delta(\mathcal{R}')\cap \Upsilon^*[Z\mapsto\psi]\neq\emptyset$. Since $\mathcal{R}'$ is from $\mathsf{ID}_{0,<0}^{\infty,+}$, the conclusion is a single formula $\upsilon[Z\mapsto \psi]$. Moreover, we cannot have $\upsilon[Z\mapsto \psi]$ be ${\sim}\psi(m')$---if that were the case, we would have had $m'\not\in\mu xX.\phi$ in $\Upsilon=\Upsilon^*[Z\mapsto\mu xX.\phi]$, and $\mathcal{R}$ could not have been a $\mathrm{Read}$ rule.

  This means that there is a rule $\mathcal{R}^{\leftarrow}$ with $\Delta(\mathcal{R}^{\leftarrow})=\{\upsilon[Z\mapsto\mu xX.\phi]\}$ and $|\mathcal{R}^{\leftarrow}|=|\mathcal{R}|$. We take $\Theta^{\mathrm{old}}_{\sigma\ulcorner\mathcal{R}\urcorner\ulcorner\mathcal{R}'\urcorner}=\Theta^{\mathrm{old}}_\sigma\cup[\Upsilon\cup\{\upsilon'_\iota[Z\mapsto\mu xX.\phi]\}]^{m\in\mu yY.\rho,\nu\iota}$ and $\Theta^{\mathrm{new}}_{\sigma\ulcorner\mathcal{R}\urcorner\ulcorner\mathcal{R}'\urcorner}=\Theta^{\mathrm{new}}_\sigma\cup[\Upsilon'\cup\{\upsilon'_\iota[Z\mapsto\psi]\}]^{m\in\mu yY.\rho,\nu\iota}$ (where $\upsilon'_\iota[Z\mapsto \mu xX.\phi]$ is the unique formula in $\Delta_\iota(\mathcal{R}^{\leftarrow})$) and we can set $\mathrm{Subst}^{n \in \mu xX. \phi\mapsto \psi(n)}(\sigma\ulcorner\mathcal{R}\urcorner\ulcorner\mathcal{R}'\urcorner)=\mathrm{Read}^{\mathsf{ID}^{\infty,+}_{0,<0}}_{[\Theta'_{\sigma\ulcorner\mathcal{R}\urcorner\ulcorner\mathcal{R}'\urcorner}]^{n\in\mu xX.\phi,\epsilon\ulcorner\mathcal{R}^{\leftarrow}\urcorner}}$.

  Next, suppose $\Delta(\mathcal{R}')\cap \Upsilon^+\neq\emptyset$. Then, since $\Upsilon^+\subseteq\Theta_\sigma^{\mathrm{new}}$, we may set $\mathrm{Subst}^{n \in \mu xX. \phi\mapsto \psi(n)}(\sigma\ulcorner\mathcal{R}\urcorner\ulcorner\mathcal{R}'\urcorner)=\mathcal{R}'$ and, for each $\iota$, set $\Theta^{\mathrm{old}}_{\sigma\ulcorner\mathcal{R}\urcorner\ulcorner\mathcal{R}'\urcorner\iota}=\Theta^{\mathrm{old}}_\sigma$ and $\mathrm{Subst}^{n \in \mu xX. \phi\mapsto \psi(n)}(\sigma\ulcorner\mathcal{R}\urcorner\ulcorner\mathcal{R}'\urcorner\iota)=\mathrm{Read}^{\mathsf{ID}^{\infty,+}_{0,<0}}_{[\Theta'_{\sigma\ulcorner\mathcal{R}\urcorner\ulcorner\mathcal{R}'\urcorner\iota}]^{n\in\mu xX.\phi,\epsilon\mathrm{Rep}}}$.

  Next, suppose $\Delta(\mathcal{R}')\cap(\Upsilon'\setminus\Upsilon^+)\neq\emptyset$---that is, we have $m\not\in\mu xX.\phi$ appearing in $\Upsilon$ while ${\sim}\psi(m)$ appears in $\Upsilon'$. We will simply copy rules introducing ${\sim}\psi(m)$, and use the side branch of the $\mathrm{Read}$ rule we are reading to cut it away. We set $\mathrm{Subst}^{n\in\mu xX.\phi\mapsto\psi(n)}(\sigma\ulcorner\mathcal{R}\urcorner\ulcorner\mathcal{R}'\urcorner)=\mathrm{Cut}_{\psi(m)}$.

  We define $\Theta^{\mathrm{old}}_{\sigma\ulcorner\mathcal{R}\urcorner\ulcorner\mathcal{R}'\urcorner\mathrm{L}}=\Theta^{\mathrm{old}}_\sigma\cup\{m\in\mu xX.\phi\}$, $\Theta^{\mathrm{new}}_{\sigma\ulcorner\mathcal{R}\urcorner\ulcorner\mathcal{R}'\urcorner\mathrm{L}}=\Theta^{\mathrm{new}}_\sigma\cup\{\psi(m)\}$, and then set $\mathrm{Subst}^{n\in\mu xX.\phi\mapsto\psi(n)}(\sigma\ulcorner\mathcal{R}\urcorner\ulcorner\mathcal{R}'\urcorner\mathrm{L})=\mathrm{Read}^{\mathsf{ID}^{\infty,+}_{0,<0}}_{[\Theta'_{\sigma\ulcorner\mathcal{R}\urcorner\mathrm{L}}]^{n\in\mu xX.\phi,\epsilon(m\in\mu xX.\phi)}}$.

  We set $\mathrm{Subst}^{n\in\mu xX.\phi\mapsto\psi(n)}(\sigma\ulcorner\mathcal{R}\urcorner\ulcorner\mathcal{R}'\urcorner\mathrm{R})=\mathcal{R}'$ and, for each $\iota$, $\Theta^{\mathrm{old}}_{\sigma\ulcorner\mathcal{R}\urcorner\ulcorner\mathcal{R}'\urcorner\mathrm{R}\iota}=\Theta^{\mathrm{old}}_\sigma$, $\Theta^{\mathrm{new}}_{\sigma\ulcorner\mathcal{R}\urcorner\ulcorner\mathcal{R}'\urcorner\mathrm{R}\iota}=\Theta^{\mathrm{new}}_\sigma\cup\{{\sim}\psi(m)\}$, and then set $\mathrm{Subst}^{n\in\mu xX.\phi\mapsto\psi(n)}(\sigma\ulcorner\mathcal{R}\urcorner\ulcorner\mathcal{R}'\urcorner\mathrm{R}\iota)=\mathrm{Read}^{\mathsf{ID}^{\infty,+}_{0,<0}}_{[\Theta'_{\sigma\ulcorner\mathcal{R}\urcorner\mathrm{Rep}}]^{n\in\mu xX.\phi,\epsilon(m\in\mu xX.\phi)}}$.

  Otherwise $\Delta(\mathcal{R}')\cap\Upsilon'=\emptyset$ and we may take $\mathrm{Subst}^{n \in \mu xX. \phi\mapsto \psi(n)}(\sigma\ulcorner\mathcal{R}\urcorner\ulcorner\mathcal{R}'\urcorner)=\mathcal{R}'$ and, for each $\iota$, set $\Theta^{\mathrm{old}}_{\sigma\ulcorner\mathcal{R}\urcorner\ulcorner\mathcal{R}'\urcorner\iota}=\Theta^{\mathrm{old}}_\sigma$ and $\mathrm{Subst}^{n \in \mu xX. \phi\mapsto \psi(n)}(\sigma\ulcorner\mathcal{R}\urcorner\ulcorner\mathcal{R}'\urcorner\iota)=\mathrm{Read}^{\mathsf{ID}^{\infty,+}_{0,<0}}_{[\Theta'_{\sigma\ulcorner\mathcal{R}\urcorner\ulcorner\mathcal{R}'\urcorner\iota}]^{n\in\mu xX.\phi,\epsilon\ulcorner\mathrm{Rep}\urcorner}}$. (That is, in this case we simply copy over the rule, since it is harmless, and we ``simulate'' a $\mathrm{Rep}$ rule on the function we have been given as an input.) If $\Delta(\mathcal{R}')\cap\Upsilon^+=\emptyset$, we take $\Theta^{\mathrm{new}}_{\sigma\ulcorner\mathcal{R}\urcorner\ulcorner\mathcal{R}'\urcorner\iota}=\Theta^{\mathrm{new}}_\sigma$; if $\Delta(\mathcal{R}')\cap\Upsilon^+\neq\emptyset$, we instead take $\Theta^{\mathrm{new}}_{\sigma\ulcorner\mathcal{R}\urcorner\ulcorner\mathcal{R}'\urcorner\iota}\cup\{[\Upsilon'\cup\Delta_\iota(\mathcal{R}')]^{m\in\mu yY.\phi'(\psi),\nu'\iota}\}$.

  Because of our modification to the Read rule, there may be additional premises indexed by a formula $m\in\mu yY. \theta$ because $m\not\in\mu yY.\theta$ appears in $\Upsilon'$. We may set $\Theta^{\mathrm{old}}_{\sigma\ulcorner\mathcal{R}\urcorner\ulcorner m\in\mu yY.\theta(\mu xX.\phi)\urcorner}=\Theta^{\mathrm{old}}_\sigma$, $\Theta^{\mathrm{new}}_{\sigma\ulcorner\mathcal{R}\urcorner\ulcorner m\in\mu yY.\theta(\psi)\urcorner}=\Theta^{\mathrm{new}}_\sigma$, and set $\mathrm{Subst}^{n\in\mu xX.\phi\mapsto\psi(n)}(\sigma\ulcorner\mathcal{R}\urcorner\ulcorner m\in\mu yY.\theta\urcorner)=\mathrm{Read}^{\mathsf{ID}^{\infty,+}_{0,<0}}_{[\Theta'_{\sigma\ulcorner\mathcal{R}\urcorner\ulcorner m\in\mu yY.\theta(\psi)\urcorner}]^{n\in\mu xX.\phi,\epsilon\ulcorner m\in\mu yY.\theta(\mu xX.\phi)\urcorner}}$.

\end{proof}

\begin{lemma}
  Let $d$ be a deduction in $\mathsf{ID}_{<\bOmega+\omega}$ so that $\Gamma(d)$ has free variables contained in $x_1,\ldots,x_n$. There are some $C_0,k,a$ so that, for any numerals $m_1,\ldots,m_n$, there is a proof tree $d^\infty$ in $\mathsf{ID}^\infty_{\bOmega+C_0,<k}$ with $\Gamma(d^\infty)\subseteq\Gamma(d)[x_i\mapsto m_i]$.
\end{lemma}

 \subsection{Ordinal Terms}

 We use the ordinal notation $\mathrm{OT}_{\omega+\Xi+\omega}$ introduced in \cite{towsner2025polymorphicordinalnotations}. This system has ``cardinal'' symbols
 \[\Omega_1<\Omega_2<\cdots<\Xi<\Omega_{\bOmega+1}<\cdots<\Omega_{\bOmega+n}.\]
 For any $C_0$, we may consider the restriction to terms in which the only $\Omega_n,\Omega_{\bOmega+n}$ appearing are those with $n<C_0$, and then extend this by variables for each cardinal symbol. We refer to this restriction as $\mathrm{OT}_{C_0+\Xi+C_0}$.

\begin{definition}
  When $d$ is a proof tree in $\mathsf{ID}^\infty_c$ with conclusion $\Gamma$ and $\gamma$ is an ordinal term, an \emph{ordinal bound above $\gamma$ on $d$} is a function $o^d:\dom(d)\rightarrow \mathrm{OT}_{C_0+\Xi+C_0}$ such that:
  \begin{itemize}
  \item for all $\sigma\in\dom(d)$, if $v_{n\in\mu xX.\phi,\epsilon}\in \mathrm{FV}(o^d(\sigma))$ then $[\Theta]^{n\in\mu xX.\phi,\epsilon}\in\Gamma(d,\sigma)$ for some $\Theta$,
  \item for all $\sigma\in\dom(d)$, if $v_{\mu xX.\phi}\in \mathrm{FV}(o^d(\sigma))$ then some $m\not\in\mu xX.\phi$ appears as a subformula of an element of $\Gamma(d,\sigma)$
  \item if $\sigma\sqsubset\tau$ then $o^d(\tau)\ll_\gamma o^d(\sigma)$.
  \end{itemize}
\end{definition}

\subsection{Bounds}

\begin{lemma}
  The proof tree $d_{n\in\mu xX.\phi}$ of $\{n\in\mu xX.\phi,n\not\in\mu xX.\phi\}$ is bounded by $\Xi^{(0)}(\Omega^{(0)}_c)$ where $c$ is the depth of $\mu yY.\phi'(\mu xX.\phi)$.
\end{lemma}
\begin{proof}
  We set $o^{d_{n\in\mu xX.\phi}}(\langle\top\rangle)=v_{n\in\mu xX.\phi,\langle\rangle}(v_{\mu yY.\phi'})\#1$. In general, at some $\sigma$ deducing
  \[[\Gamma,[\Delta_1]^{m_1\in\mu yY.\phi',\epsilon_1},\ldots, [\Delta_k]^{m_k\in\mu yY.\phi',\epsilon_k}]^{n\in\mu xX.\phi,\epsilon},\Gamma, [\Delta_1]^{m_1\in\mu yY.\phi',\epsilon_1},\ldots, [\Delta_k]^{m_k\in\mu yY.\phi',\epsilon_k},\]
  we set $o^{d_{n\in\mu xX.\phi}}(\sigma)=v_{n\in\mu xX.\phi,\epsilon}(v_{\mu yY.\phi'},v_{m_1\in\mu yY.\phi',\epsilon_1},\ldots, v_{m_k\in\mu yY.\phi',\epsilon_k})$.
\end{proof}

\begin{lemma}
  For every proper closed formula $\phi$, there is a proof tree $d_\phi$ in $\mathsf{ID}^\infty_0$ with $\Gamma(d_\phi)\subseteq\{\phi,{\sim}\phi\}$ bounded by $\#_{0< c<C_0}\{\Omega_{c},\Omega^{(0)}_{\bOmega+c},\Xi^{(0)}(\Omega^{(0)}_{\bOmega+c})\}\#rk(\phi)$.
\end{lemma}

The same bound on the substitution function works---for successor $c$, the substitution function is bounded by
\[\#_{0< c<C_0}\{\Omega_{c},\Omega^{(0)}_{\bOmega+c},\Xi^{(0)}(\Omega^{(0)}_{\bOmega+c})\}\#rk(\psi)\cdot 2\#v_{n\in\mu xX.\phi,\langle\rangle}\#2\]
while the new substitution function for depth $\bOmega$ is bounded by
  \[\#_{0< c<C_0}\{\Omega_{c},\Omega^{(0)}_{\bOmega+c},\Xi^{(0)}(\Omega^{(0)}_{\bOmega+c})\}\#rk(\psi)\#w_{n\in\mu xX.\phi,\langle\rangle}(\Omega^{(0)}_c)\#2\]
where $c$ is the depth of $\mu yY.\phi'(\psi)$.

\begin{lemma}
  Let $d$ be a deduction in $\mathsf{ID}_{<\bOmega+\omega}$ so that $\Gamma(d)$ has free variables contained in $x_1,\ldots,x_n$. There are some $C_0,k,a$ so that, for any numerals $m_1,\ldots,m_n$, there is a proof tree $d^\infty$ in $\mathsf{ID}^\infty_{\bOmega+C_0,<k}$ with $\Gamma(d^\infty)\subseteq\Gamma(d)[x_i\mapsto m_i]$ and bounded by
  \[\#_{0< c<C_0}\{\Omega_{c},\Omega^{(0)}_{\bOmega+c},\Xi^{(0)}(\Omega^{(0)}_{\bOmega+c})\}\cdot 2\#\omega\#a.\]
\end{lemma}

The bounds on the cut elimination operations are unchanged. (Indeed, the cut elimination operations can, themselves, be written as proof trees with the usual bounds and then extended to suitable domains by lifting, so in a precise sense, the cut elimination operations on $\mathsf{ID}_{<\bOmega+C_0}$ are identical to those on smaller systems.)

Collapsing at successors is also essentially unchanged.
\begin{lemma}
  Suppose $d$ is a proof tree in $\mathsf{ID}_{c+1,<0}^\infty$ with $c$ either $0$ or a successor bounded by $\alpha$ above $\gamma$. Let $\beta=\Omega_{c+1}\#{\scalebox{1.5}{\#}}\{v_{\mu xX.\phi}\#1\}$ where $\#$ ranges over those $\mu xX.\phi$ appearing negatively in $\Gamma(d)$.  Then there is a proof tree $D_{c+1}(d)$ in $\mathsf{ID}^\infty_{c,<0}$ with $\Gamma(D_{c+1}(d))\subseteq\Gamma(d)$ and bounded by $\vartheta_{c+1,\gamma}(\alpha[\Omega_{c+1}\mapsto\beta])$.
\end{lemma}

\begin{theorem}
  Suppose $d$ is a proof tree in $\mathsf{ID}_{\bOmega+1,<0}^\infty$ bounded by $\alpha$ above $\gamma$ whose conclusion does not contain negative occurences of $\mu$-expressions of depth $\bOmega$. Then there is a proof tree $D_{\bOmega}(d)$ in $\mathsf{ID}^{\infty}_{C_0-1,<0}$ with $\Gamma(D_{\bOmega}(d))\subseteq\Gamma(d)$ and bounded by $\vartheta_{\bOmega,\gamma}(\alpha)$.
\end{theorem}
\begin{proof}
  The argument is essentially the same as before. Once again we define $D_{\bOmega}(d)(\sigma)$ by induction on $|\sigma|$ for all $d$ simultaneously along with the partial function $\pi_d:\dom(D_{\bOmega}(d))\rightarrow\dom(d)$. We can then define the ordinal assignment $o^{D_{\bOmega}(d)}:\dom(D_{\bOmega}(d))\rightarrow\dom(d)$. Naturally we begin by setting $\pi_d(\langle\rangle)=\langle\rangle$.

  As long as $d(\pi_d(\sigma))$ is anything other than a $\mathrm{Cut}\Omega^\flat_{n\in\mu xX.\phi}$ where $n\in\mu xX.\phi$ has depth $\bOmega$, we may set $D_{\bOmega}(d)(\sigma)=d(\pi_d(\sigma))$ and set $\pi_d(\sigma\iota)=\pi_d(\sigma)\iota$.

  The main case is when $d(\pi_d(\sigma))$ is a $\mathrm{Cut}\Omega^\flat_{n\in\mu xX.\phi}$ where $n\in\mu xX.\phi$ has depth $\bOmega$. As always, we set $D_{\bOmega}(d)(\sigma)=\mathrm{Rep}$ and $D_{\bOmega}(d)(\sigma\top\upsilon)=D_{\bOmega}(\overline{d}^{\uparrow \mathsf{ID}^{\infty}_{C_0-1,<0}}_{\pi_d(\sigma)\bot}(D_{\bOmega}(d_{\pi_d(\sigma)\top})))(\upsilon)$.

  When $\pi_d(\sigma)$ is defined, we set $o^{D_{\bOmega}(d)}(\sigma)=\vartheta_{\bOmega,\gamma}(o^d(\pi_d(\sigma)))$. Otherwise we copy over $o^{D_{\bOmega}(d)}(\sigma\top\upsilon)=o^{D_{\bOmega}(\overline{d}^{\uparrow \mathsf{ID}^{\infty}_{C_0-1,<0}}_{\pi_d(\sigma)\bot}(D_{\bOmega}(d_{\pi_d(\sigma)\top})))}(\upsilon)$.

  When $d(\pi_d(\sigma))$ is not $\mathrm{Cut}\Omega^\flat_{n\in\mu xX.\phi}$, we have $o^{D_d(d)}(\sigma\iota)=\vartheta_{\bOmega,\gamma}o^d(\pi_d(\sigma\iota))$, and since $o^d(\pi_d(\sigma\iota))\ll_\gamma o^d(\pi_d(\sigma))$, we have $o^{D_d(d)}(\sigma\iota)\ll_0\vartheta_{\bOmega,\gamma}o^d(\pi_d(\sigma))=o^{D_d(d)}(\sigma)$.

  When $d(\pi_d(\sigma))$ is $\mathrm{Cut}\Omega^\flat_{n\in\mu xX.\phi}$ we have
  \[o^{D_d}(\sigma\top)=o^{D_{\bOmega}(\overline{d}^{\uparrow \mathsf{ID}^{\infty}_{C_0-1,<0}}_{\pi_d(\sigma)\bot}(D_{\bOmega}(d_{\pi_d(\sigma)\top})))}(\langle\rangle)=\vartheta_{\bOmega,\vartheta_{\bOmega,\gamma}(o^d(\pi_d(\sigma)\top))}(o^d(\pi_d(\sigma)\top)[v_{n\in\mu xX.\phi,\langle\rangle}\mapsto \vartheta_{\bOmega,\gamma}(o^d(\pi_d(\sigma)\top))]).\]
Since $o^d(\pi_d(\sigma)\top)\ll_{\gamma} o^d(\pi_d(\sigma))$ and $o^d(\pi_d(\sigma)\bot)\ll_{\bOmega,\gamma} o^d(\pi_d(\sigma))$, we have $o^{D_d}(\sigma\top)\ll_0o^{D_d}(\sigma)$ as needed.

\end{proof}

\begin{theorem}
  If $d$ is a proof in $\mathsf{ID}_{\bOmega+\omega}$ whose conclusion does not contain negative occurences of any $\mu xX.\phi$ then there is a proof tree $d'$ in $\mathsf{ID}_{0,<0}^{\infty,+}$ with $\Gamma(d')\subseteq\Gamma(d)$ bounded by $\vartheta_{0,0}(\omega_k^{\#_{0< c<C}\{\Omega_{c},\Omega^{(0)}_{\bOmega+c},\Xi^{(0)}(\Omega^{(0)}_{\bOmega+c})\}\cdot 2\#\omega\#a})$ for some finite $C,k,a$.
\end{theorem}

\printbibliography
\end{document}